\documentclass{amsart}
\usepackage{latexsym, amssymb, amsmath, amsthm}
\usepackage[T1]{fontenc}
\usepackage{lmodern}
\usepackage{enumerate}
\usepackage{mathabx}
\usepackage{csquotes}
\usepackage[mathscr]{euscript}

\newcommand\blfootnote[1]{%
  \begingroup
  \renewcommand\thefootnote{}\footnote{#1}%
  \addtocounter{footnote}{-1}%
  \endgroup
}

\usepackage{tikz}
\usetikzlibrary{trees}

\makeatletter
\newlength\knuthian@fdfive
\def\mathpal@save#1{\let\was@math@style=#1\relax}
\def\utilde#1{\mathpalette\mathpal@save
              {\setbox124=\hbox{$\was@math@style#1$}%
\setbox125=\hbox{$\fam=3\global\knuthian@fdfive=\fontdimen5\font$}
\setbox125=\hbox{$\widetilde{\vrule height 0pt depth 0pt width \wd124}$}%
               \baselineskip=1pt\relax
               \lineskiplimit=\z@\relax
               \lineskip=1pt\relax
               \vtop{\copy124\copy125\vskip -\knuthian@fdfive}}}
\makeatother

\usepackage{pifont}

\newtheorem{theorem}{Theorem}[section]
\newtheorem{lemma}[theorem]{Lemma}
\newtheorem{corollary}[theorem]{Corollary}
\newtheorem{proposition}[theorem]{Proposition}
\newtheorem{question}[theorem]{Question}
\newtheorem*{claim}{Claim}

\theoremstyle{definition}
\newtheorem{definition}[theorem]{Definition}
\theoremstyle{remark}
\newtheorem{remark}[theorem]{Remark}

\newcommand{\Ref}[2]{\mathsf{RFN}_{#1}}

\newbox\gnBoxA
\newdimen\gnCornerHgt
\setbox\gnBoxA=\hbox{$\ulcorner$}
\global\gnCornerHgt=\ht\gnBoxA
\newdimen\gnArgHgt
\def\gnmb #1{%
\setbox\gnBoxA=\hbox{$#1$}%
\gnArgHgt=\ht\gnBoxA%
\ifnum     \gnArgHgt<\gnCornerHgt \gnArgHgt=0pt%
\else \advance \gnArgHgt by -\gnCornerHgt%
\fi \raise\gnArgHgt\hbox{$\ulcorner$} \box\gnBoxA %
\raise\gnArgHgt\hbox{$\urcorner$}}

\title{Reflection ranks and ordinal analysis}
\author{Fedor Pakhomov}\thanks{The first author is supported in part by Young Russian Mathematics award}
\address{Steklov Mathematical Institute, Moscow\vspace{-8pt}}
\address{Institute of Mathematics of the Czech Academy of Sciences}
\email{pakhfn@mi-ras.ru}

\author{James Walsh}
\address{Group in Logic and the Methodology of Science, University of California, Berkeley}
\email{walsh@math.berkeley.edu}

\begin{document}

\begin{abstract}
It is well-known that natural axiomatic theories are well-ordered by consistency strength. However, it is possible to construct descending chains of artificial theories with respect to consistency strength. We provide an explanation of this well-orderedness phenomenon by studying a coarsening of the consistency strength order, namely, the $\Pi^1_1$ reflection strength order. We prove that there are no descending sequences of $\Pi^1_1$ sound extensions of $\mathsf{ACA}_0$ in this ordering. Accordingly, we can attach a rank in this order, which we call reflection rank, to any $\Pi^1_1$ sound extension of $\mathsf{ACA}_0$. We prove that for any $\Pi^1_1$ sound theory $T$ extending $\mathsf{ACA}_0^+$, the reflection rank of $T$ equals the proof-theoretic ordinal of $T$. We also prove that the proof-theoretic ordinal of $\alpha$ iterated $\Pi^1_1$ reflection is $\varepsilon_\alpha$. Finally, we use our results to provide straightforward well-foundedness proofs of ordinal notation systems based on reflection principles.

\end{abstract}

\blfootnote{Thanks to Lev Beklemishev and Antonio Montalb\'{a}n for helpful discussion and to an anonymous referee for many useful comments and suggestions.}

\maketitle

\section{Introduction}

It is a well-known empirical phenomenon that \emph{natural} axiomatic theories are well-ordered\footnote{Of course, by \emph{well-ordered} here we mean \emph{pre-well-ordered}.} according to many popular metrics of proof-theoretic strength, such as consistency strength. This phenomenon is manifest in \emph{ordinal analysis}, a research program wherein recursive ordinals are assigned to theories to measure their proof-theoretic strength. However, these metrics of proof-theoretic strength do \emph{not} well-order axiomatic theories \emph{in general}. For instance, there are descending chains of sound theories, each of which proves the consistency of the next. However, all such examples of ill-foundedness make use of unnatural, artificial theories. Without a mathematical definition of ``natural,'' it is unclear how to provide a general mathematical explanation of the apparent well-orderedness of the hierarchy of natural theories.

In this paper we introduce a metric of proof-theoretic strength and prove that it is immune to these pathological instances of ill-foundedness. Recall that a theory $T$ is $\Pi^1_1$ sound just in case every $\Pi^1_1$ theorem of $T$ is true. The $\Pi^1_1$ soundness of $T$ is expressible in the language of second-order arithmetic by a formula $\mathsf{RFN}_{\Pi^1_1}(T)$. The formula $\mathsf{RFN}_{\Pi^1_1}(T)$ is also known as the \emph{uniform $\Pi^1_1$ reflection principle for $T$.}

\begin{definition}
For theories $T$ and $U$ in the language of second-order arithmetic we say that $T\prec_{\Pi^1_1}U$ if $U$ proves the $\Pi^1_1$ soundness of $T$.
\end{definition}

This metric of proof-theoretic strength is coarser than consistency strength, but, as we noted, it is also more robust. In practice, when one shows that $U$ proves the consistency of $T$, one often also establishes the stronger fact that $U$ proves the $\Pi^1_1$ soundness of $T$. Our first main theorem is the following.

\begin{theorem}
\label{first}
The restriction of $\prec_{\Pi^1_1}$ to the $\Pi^1_1$-sound extensions of $\mathsf{ACA}_0$ is well-founded.
\end{theorem}

Accordingly, we can attach a well-founded rank---\emph{reflection rank}---to $\Pi^1_1$ sound extensions of $\mathsf{ACA}_0$ in the $\prec_{\Pi^1_1}$ ordering.

\begin{definition} 
The \emph{reflection rank} of $T$ is the rank of $T$ in the ordering $\prec_{\Pi^1_1}$ restricted to $\Pi^1_1$ sound extensions of $\mathsf{ACA}_0$. We write $|T|_{\mathsf{ACA}_0}$ to denote the reflection rank of $T$.
\end{definition}

What is the connection between the reflection rank of $T$ and the $\Pi^1_1$ proof-theoretic ordinal of $T$? Recall that the $\Pi^1_1$ \emph{proof-theoretic ordinal} $|T|_{\mathsf{WO}}$ of a theory $T$ is the supremum of the order-types of $T$-provably well-founded primitive recursive linear orders. We will show that the reflection ranks and proof-theoretic ordinals of theories are closely connected. Recall that $\mathsf{ACA}_0^+$ is axiomatized over $\mathsf{ACA}_0$ by the statement ``for every $X$, the $\omega^{\mathrm{th}}$ jump of $X$ exists.''

\begin{theorem}
\label{second}
For any $\Pi^1_1$-sound extension $T$ of $\mathsf{ACA}^+_0$, $|T|_{\mathsf{ACA}_0}=|T|_{\mathsf{WO}}$.
\end{theorem}

In general, if $|T|_{\mathsf{ACA}_0} = \boldsymbol\alpha$ then $|T|_{\mathsf{WO}} \geq \boldsymbol{\varepsilon_\alpha}$. We provide examples of theories such that $|T|_{\mathsf{ACA}_0} = \boldsymbol\alpha$ and $|T|_{\mathsf{WO}} >\boldsymbol{\varepsilon_\alpha}$. Nevertheless for many theories $T$ with $|T|_{\mathsf{ACA}_0} = \boldsymbol\alpha$ we have $|T|_{\mathsf{WO}} = \boldsymbol{\varepsilon_\alpha}$.

To prove these results, we extend techniques from the proof theory of iterated reflection principles to the second-order context. In particular, we focus on iterated $\Pi^1_1$ reflection. Roughly speaking, the theories $\mathbf{R}^{\alpha}_{\Pi^1_1}(T)$ of $\alpha$-iterated $\Pi^1_1$-reflection over $T$ are defined as follows
\begin{flalign*}
\mathbf{R}^0_{\Pi^1_1}(T) &: = T \\
\mathbf{R}^\alpha_{\Pi^1_1}(T) &: = T + \bigcup_{\beta\prec\alpha}\mathsf{RFN}_{\Pi^1_1}\big(\mathbf{R}^\beta_{\Pi^1_1}(T)\big)\textrm{ for $\alpha\succ 0$.}
\end{flalign*}
The formalization of this definition in arithmetic requires some additional efforts; see \textsection{2} for details.

Iterated reflection principles have been used previously to calculate proof-theoretic ordinals. For instance, Schmerl \cite{schmerl1979fine} used iterated reflection principles to establish bounds on provable arithmetical transfinite induction principles for fragments of $\mathsf{PA}$. Beklemishev \cite{beklemishev2003proof} has also calculated proof-theoretic ordinals of subsystems of $\mathsf{PA}$ via iterated reflection. These results differ from ours in two important ways. First, these results concern only theories in the language of first-order arithmetic, and hence do not engender calculations of $\Pi^1_1$ proof-theoretic ordinals. Second, these results are notation-dependent, i.e., they involve the calculation of proof-theoretic ordinals \emph{modulo} the choice of a particular (natural) ordinal notation system. We are concerned with $\Pi^1_1$ reflection. Hence, in light of Theorem \ref{first}, we are able to calculate proof-theoretic ordinals in a manner that is not sensitive to the choice of a particular ordinal notation system. 

\begin{theorem}
\label{iterated}Let $\alpha$ be an ordinal notation system with the order type $|\alpha|=\boldsymbol\alpha$. Then 
$|\mathbf{R}^\alpha_{\Pi^1_1}(\mathsf{ACA}_0)|_{\mathsf{ACA}_0}=\boldsymbol\alpha$ and $|\mathbf{R}^\alpha_{\Pi^1_1}(\mathsf{ACA}_0)|_{\mathsf{WO}}=\boldsymbol{\varepsilon_\alpha}$.
\end{theorem}

It is possible to prove Theorem \ref{second} and Theorem \ref{iterated} by formalizing infinitary derivations in $\mathsf{ACA}_0$ and appealing to cut-elimination, and in an early draft of this paper we did just that. Lev Beklemishev suggested that it might be possible to prove these results with methods from the proof theory of iterated reflection principles, namely conservation theorems in the style of Schmerl \cite{schmerl1979fine}. Though these methods have become quite polished for studying subsystems of first-order arithmetic, they have not yet been extended to $\Pi^1_1$ ordinal analysis. Thus, we devote a section of the paper to developing these techniques in the context of second-order arithmetic. We thank Lev for encouraging us to pursue this approach. Our main result in this respect is the following conservation theorem, where $\Pi^1_1(\Pi^0_3)$ denotes the complexity class consisting of formulas of the form $\forall X \;\mathsf{F}$ where $\mathsf{F}\in\Pi^0_3$.

\begin{theorem}
\label{main tool}
$\mathbf{R}^\alpha_{\Pi^1_1}(\mathsf{ACA}_0)$ is $\Pi^1_1(\Pi^0_3)$ conservative over $\mathbf{R}^{\varepsilon_\alpha}_{\Pi^1_1(\Pi^0_3)}(\mathsf{RCA}_0)$. 
\end{theorem}

To prove this result, we establish connections between $\Pi^1_1$ reflection over second-order theories and reflection over arithmetical theories with free set variables.
 
Finally, we demonstrate that Theorem \ref{first} could be used for straightforward well-foundedness proofs for certain ordinal notation systems. A recent development in ordinal analysis is the use of ordinal notation systems that are based on reflection principles. Roughly, the elements of such notation systems are reflection principles and they are ordered by proof-theoretic strength. Such notation systems have been extensively studied since Beklemishev \cite{beklemishev2004provability} endorsed their use as an approach to the canonicity problem for ordinal notations. See \cite{fernandez2016worms} for a survey of such notation systems.  We prove the well-foundedness of Beklemishev's reflection notation system for $\varepsilon_0$ using the well-foundedness of the $\prec_{\Pi^1_1}$-order. Previously, Beklemishev proved the well-foundedness of this system by constructing the isomorphism with Cantor's ordinal notation system for $\varepsilon_0$. We expect that our techniques---or extensions thereof---could be used to prove the well-foundedness of ordinal notation systems for stronger axiomatic theories.

Here is our plan for the rest of the paper. In \textsection{\ref{definitions}} we fix our notation and introduce some key definitions. In \textsection{\ref{dssection}} we present our technique for showing that certain classes of theories are well-founded (or nearly well-founded) according to various notions of proof-theoretic strength. Our first application of this technique establishes Theorem \ref{first}. In \textsection{\ref{dssection}} we prove various conservation results that connect iterated reflection principles with transfinite induction. The theorems in \textsection{\ref{dssection}} extend results of Schmerl from first-order theories to pseudo $\Pi^1_1$ theories, i.e., to theories axiomatized by formulas with at most free set variables, and to second-order theories. We conclude with a proof of Theorem \ref{main tool}. In \textsection{\ref{conservativity_section}} we establish connections between the reflection ranks and proof-theoretic ordinals of theories, including proofs of Theorem \ref{second} and Theorem \ref{iterated}. In \textsection{\ref{proof-theoretic_ordinal}} we show how to use our results to prove the well-foundedness of ordinal notation systems based on reflection principles. In \textsection\ref{ref_not_sect} we present an explicit example by proving the well-foundedness of Beklemishev's notation system for $\varepsilon_0$.


\section{Definitions and notation}
\label{definitions}

In this section we describe and justify our choice of meta-theory. We then fix some notation and present some key definitions. Finally, we describe a proof-technique that we will use repeatedly throughout the paper, namely, Schmerl's technique of reflexive induction.

\subsection{Treatment of theories}

Recall that $\mathsf{EA}$ is a finitely axiomatizable theory in the language of arithmetic with the exponential function, i.e., in the signature $(0,1,+,\times,2^x,\leq)$. $\mathsf{EA}$ is characterized by the standard recursive axioms for addition, multiplication, and exponentiation as well as the induction schema for $\Delta_0$ formulas. Note that by $\Delta_0$ formulas we mean $\Delta_0$ formulas in the language with exponentiation. $\mathsf{EA}$ is strong enough to facilitate typical approaches to arithmetization of syntax. Moreover, $\mathsf{EA}$ proves its own $\Sigma_1$ completeness.

We will also be interested in $\mathsf{EA}^+$. $\mathsf{EA}^+$ is a theory in the language of $\mathsf{EA}$. $\mathsf{EA}^+$ extends $\mathsf{EA}$ by the additional axiom ``superexponentiation is total.'' By superexponentiation, we mean the function $2^x_x$ where $2^x_0=x$ and $2^x_{y+1}=2^{2^x_y}$. $\mathsf{EA}^+$ is the weakest extension of $\mathsf{EA}$ in which the cut-elimination theorem is provable. Indeed, the cut-elimination theorem is equivalent to the totality of superexponentiation over $\mathsf{EA}$. See \cite{beklemishev2005reflection} for details on $\mathsf{EA}$ and $\mathsf{EA}^+$; see also \cite{hajek2017metamathematics} for details on $\mathsf{EA}$ and $\mathsf{EA}^+$ in a slightly different formalism without an explicit symbol for exponentiation. We will use $\mathsf{EA}^+$ as a meta-theory for proving many of our results.

In this paper we will examine theories in three different languages. First the language of first-order arithmetic, i.e., the language of $\mathsf{EA}$. Second the language of first-order arithmetic extended with one additional free set variable $X$; we also call this the \emph{pseudo-$\Pi^1_1$ language}. And finally the language of second-order arithmetic. The language of first-order arithmetic of course is a sublanguage of the other two languages. And we consider the pseudo-$\Pi^1_1$ language to be a sublanguage of the language of second-order arithmetic by identifying each pseudo-$\Pi^1_1$ sentence $\mathsf{F}$ with the second-order sentence $\forall X \;\mathsf{F}$. 

In the first-order context we are interested in the standard arithmetical complexity classes $\Pi_n$ and $\Sigma_n$. We write $\Pi_\infty$ to denote the class of all arithmetical formulas. 
We write $\mathbf{\Pi}^0_n$ to denote the class of formulas that are just like $\Pi_n$ formulas except that their formulas (potentially) contain a free set variable $X$. Formulas in the complexity class $\mathbf{\Pi}^0_n$ \emph{cannot} have set quantifiers, and so contain \emph{only} free set variables. Of course, the class $\mathbf{\Sigma}^0_n$ is defined dually to the class $\mathbf{\Pi}^0_n$. We write $\mathbf{\Pi}^0_\infty$ to denote the class of boldface arithmetical formulas, i.e., the class of arithmetical formulas (potentially) with a free set variable.

In the second-order context we are mostly interested in the standard analytical complexity classes $\Pi^1_1$ and $\Sigma^1_1$. However, we will also use other complexity classes. Suppose $\mathcal{C}\subset \mathcal{L}_2$ is one of the following classes of formulas: $\Pi^0_m$ or $\Sigma^0_m$, for $m\ge 1$. Then we denote by $\Pi^1_n(\mathcal{C})$ the class of all the formulas of the form $\forall X_1\exists X_2\ldots Q X_n \; \mathsf{F}$, where $\mathsf{F}\in\mathcal{C}$. We define $\Sigma^1_n(\mathcal{C})$ dually.

For a first-order theory $T$, we use $T(X)$ to denote the pseudo $\Pi^1_1$ pendant of $T$. For example, the theory $\mathsf{PA}(X)$ contains (i) the axioms of $\mathsf{PA}$ and (ii) induction axioms for all formulas in the language, including those with free set variables. The theories $\mathsf{EA}(X)$, $\mathsf{EA}^+(X)$, and $\mathsf{I\Sigma}_1(X)$ are defined analogously, i.e., their induction axioms are extended to include formulas with the free set variable $X$.

Formulas in any of the three languages we are working with can naturally be identified with words in a suitable finite alphabet, which, in turn, are naturally one-to-one encoded by numbers. Accordingly, we can fix a G\"{o}del numbering of these languages. We denote the G\"{o}del number of an expression $\tau$ by $\ulcorner \tau \urcorner$. Many natural syntactic relations ($x$ is a logical axiom, $z$ the result of applying Modus Ponens to $x$ and $y$, $x$ encodes a $\Pi_n$ formula, etc.) are elementary definable and their simplest properties can be verified within $\mathsf{EA}$. We also fix a one-to-one elementary coding of finite sequences of natural numbers. $\langle x_1,...,x_n\rangle$ denotes the code of a sequence $x_1,...,x_n$ and, for any fixed $n$, is an elementary function of $x_1,...,x_n$.

We are concerned with recursively enumerable theories. Officially, a theory $T$ is a $\Sigma_1$ formula $\mathsf{Ax}_{T}(x)$ that is understood as a formula defining the (G\"{o}del numbers of) axioms of $T$ in the standard model of arithmetic, i.e., the set of axioms of $T$ is $\{\varphi : \mathbb{N}\vDash \mathsf{Ax}_{T}(\varphi) \}$. Thus, we are considering theories \emph{intensionally}, via their axioms, rather than as deductively closed sets of formulas. 

Since our base theory $\mathsf{EA}$ is fairly weak, we have to be careful with our choice of formalizations of proof predicates. Namely, we want our provability predicate to be $\Sigma_1$. And due to this we can't use the straightforwardly defined predicates $\mathsf{PrfNat}_{T}(x,y)$: $x$ is a Hilbert-style proof of $y$, where all axioms are either axioms of first-order logic or axioms of $T$. The predicates $\mathsf{PrfNat}_{T}(x,y)$ are equivalent to $\forall^b\Sigma_1$-formulas over $\mathsf{EA}$ ($\forall^b\Sigma_1$-formulas are the formulas starting with a bounded universal quantifier followed by $\Sigma_1$-formula). However, $\mathsf{EA}$ is too weak to equivalently transform $\forall^b\Sigma_1$-formulas to $\Sigma_1$-formulas; for this one needs the collection scheme $\mathsf{B}\Sigma_1$, which isn't provable in $\mathsf{EA}$. We note that this doesn't affect most natural theories $T$, in particular, for any $T$ with $\Delta_0$ formula $\mathsf{Ax}_T$, the predicate $\mathsf{PrfNat}_{T}(x,y)$ is equivalent to a $\Sigma_1$ formula over $\mathsf{EA}$.

Nevertheless, to avoid this issue, we work with proof predicates that are forced to be $\Sigma_1$ in $\mathsf{EA}$, which are sometimes called  \emph{smooth proof} predicates. In the definition of the smooth proof predicate, a ``proof'' is a pair consisting of an actual Hilbert style proof and a uniform bound for witnesses to the facts that axioms in the proof indeed are axioms. We simply write $\mathsf{Prf}_{T}(x,y)$ to formalize that $x$ is a ``smooth proof'' of $y$ in theory $T$. The predicates $\mathsf{Prf}_{T}(x,y)$ are  $\Delta_0$-formulas. The predicate $\mathsf{Pr}_T(y)$ is shorthand for $\exists x \mathsf{Prf}_T(x,y)$. We use the predicate $\mathsf{Con}(T)$ as shorthand for $\neg \mathsf{Pr}_T(\bot)$, where we fix $\bot$ to be some contradictory sentence.

The closed term $1+1+...+1$ ($n$ times) is the numeral of $n$ and is denoted $\underline{n}$. We often omit the bar when no confusion can occur. We also often omit the corner quotes from G\"{o}del numbers when no confusion can occur. For instance, we can encode the notion of a formula $\varphi$ being provable in a theory $T$, by saying that there is a $T$-proof (a sequence subject to certain constraints) the last element of which is the numeral of the G\"{o}del number of $\varphi$. However, instead of writing $\mathsf{Pr}_T(\underline{\ulcorner \varphi \urcorner})$ to say that $\varphi$ is provable we simply write $\mathsf{Pr}_T( \varphi )$.

Suppose $T$ and $U$ are recursively enumerable theories in the same language. We write $T\sqsubseteq U$ if $T$ is a subtheory of $U$; we can formalize the claim that $T\sqsubseteq U$ in arithmetic with the formula $\forall \varphi \big(\mathsf{Pr}_{T}(\varphi) \rightarrow \mathsf{Pr}_U(\varphi)\big)$. We write $T \equiv U$ if $T\sqsupseteq U$ and $U\sqsupseteq T$. For a class $\mathcal{C}$ of sentences of the language of $T$ we write $T\sqsubseteq_{\mathcal{C}} U$ if the set of $\mathcal{C}$-theorems of $T$ is a subset of $\mathcal{C}$-theorems of $\mathsf{U}$; this could be naturally formalized in arithmetic with the formula $\forall \varphi \in \mathcal{C} \big(\mathsf{Pr}_{T}(\varphi) \rightarrow \mathsf{Pr}_U(\varphi)\big)$.  We write $T\equiv_{\mathcal{C}} U$ if $T\sqsubseteq_{\mathcal{C}} U$ and $U\sqsubseteq_{\mathcal{C}} T$.

We will be interested in partial truth-definitions for various classes of formulas for which we could prove Tarski's bi-conditionals. For a class $\mathcal{C}$ of formulas we call a formula $\mathsf{Tr}_{\mathcal{C}}(x)$ a partial truth definition for $\mathcal{C}$ over a theory $T$, if $\mathsf{Tr}_{\mathcal{C}}(x)$ is from the class $\mathcal{C}$ and $$T\vdash \varphi(\vec{x})\leftrightarrow \mathsf{Tr}_{\mathcal{C}}(\varphi(\vec{x}))\mbox{, for all $\varphi(\vec{x})$ from $\mathcal{C}$.}$$
Moreover, we will work only with truth definitions such that the above property is provable in $\mathsf{EA}$.

In the book by  H\'ajek and Pudl\'ak \cite[\textsection I.1(d)]{hajek2017metamathematics} there is a construction of partial truth definitions for classes $\Pi_n$ and $\Sigma_n$, $n\ge 1$, over $\mathsf{I}\Sigma_1$. However, we will use a sharper construction of partial truth definitions for classes $\Pi_n$ and $\Sigma_n$, $n\ge 1$, over $\mathsf{EA}$ which  could be found in \cite[Appendix~A]{beklemishev2019reflection}.
And we will use truth definitions for classes $\mathbf{\Pi}_n$ and $\mathbf{\Sigma}_n$,  $n\ge 1$, over $\mathsf{EA}(X)$ that as well were constructed in \cite[Appendix~A]{beklemishev2019reflection}.

In the case of second-order arithmetic there are partial truth definitions for  classes $\Pi^1_n(\Pi^0_m)$, $\Sigma^1_n(\Sigma^0_m)$, $\Sigma^1_n(\Pi^0_m)$, and $\Pi^1_n(\Sigma^0_m)$, where $m\ge 1$, over $\mathsf{RCA}_0$. One could easily construct this partial truth definitions from the partial  truth definitions for classes $\mathbf{\Pi}_n$ and $\mathbf{\Sigma}_n$ over $\mathsf{EA}(X)$. However, over $\mathsf{ACA}_0$ it is possible to construct partial truth definitions for the classes $\Pi^1_n$ and $\Sigma^1_n$, $n\ge 1$. Let $\mathbf{\Sigma}^1_1$ be the class of $\Sigma^1_1$-formulas with a set parameter $X$. It is easy to construct partial truth definitions for  classes $\Pi^1_n$ and $\Sigma^1_n$, $n\ge 1$, from a partial truth definition for $\mathbf{\Sigma}^1_1$. Simpson \cite[Lemma~V.1.4]{simpson2009subsystems} proves that for each $\Sigma^1_1$ formula $\varphi(X)$ there exists a $\Delta^0_0$ formula $\theta(x,y)$ such that
$$\mathsf{ACA}_0\vdash \forall X\big(\varphi(X)\mathrel{\leftrightarrow} (\exists f\colon \mathbb{N}\to\mathbb{N})\;\forall m\; \theta_{\varphi}(X\upharpoonright m,f \upharpoonright m)\big).$$
Here $X\upharpoonright m$ is the natural number encoding the finite set $X\cap \{0,\ldots,m-1\}$ and $f\upharpoonright m$ is the code of the finite sequence $\langle f(0),\ldots,f(m-1)\rangle$. From Simpson's proof it is easy to extract a Kalmar elementary algorithm for constructing the formula $\theta_{\varphi}$ from a formula $\varphi$. And by the same argument as Simpson we show that the $\Sigma^1_1$-formula $(\exists f\colon \mathbb{N}\to\mathbb{N})\;\forall m\; \mathsf{Tr}_{\Pi^0_1}(\theta_{x}(X\upharpoonright m,f \upharpoonright m))$ is a partial truth definition $\mathsf{Tr}_{\mathbf{\Sigma}^1_1}(X,x)$  for the class $\mathbf{\Sigma}^1_1$ over $\mathsf{ACA}_0$.



\subsection{Ordinal notations}\label{ordinal_notations}

There are many ways of treating ordinal notations in arithmetic. We choose one specific method that will be suitable when we work in the theory $\mathsf{EA}^+$ (and its extensions). Our results will be valid for other natural choices of treatment of ordinal notations, but some of the proofs would have to be tweaked slightly.

Often we will use ordinal notation systems within formal theories that couldn't prove (or even express) the well-foundedness of the notation systems. Also, most of our results are intensional in nature and don't require the notation system to be well-founded from an external point of view. Due to this, our definition of an ordinal notation system does not require it to be well-founded. 

Officially, an ordinal notation $\alpha$ is a tuple $\langle \varphi(x),\psi(x,y), p, n \rangle $ where $\varphi,\psi \in \Delta_0$, $\varphi(n)$ is true according to $\mathsf{Tr}_{\Sigma_1}$, and $p$ is an $\mathsf{EA}$ proof of the fact that on the set $\{x\mid \varphi(x)\}$ the order $$x\prec_{\alpha}y\stackrel{\mbox{\footnotesize \textrm{def}}}{\iff} \psi(x,y)$$ is a strict linear order. More formally $p$ is an $\mathsf{EA}$ proof of the conjunction of the following sentences:
\begin{enumerate}
\item \label{not_prop_1} $\forall x,y,z \big(\varphi(x)\land \varphi(y)\land\varphi(z)\land \psi(x,y)\land \psi(y,z)\to \psi(x,z)\big)$ (Transitivity);
\item \label{not_prop_2}$\forall x \big(\varphi(x)\to \neg \psi(x,x)\big)$ (Irreflexivity);
\item \label{not_prop_3}$\forall x,y \Big(\varphi(x)\land \varphi(y)\land x\ne y \to \big(\psi(x,y)\land \neg \psi(y,x)\big)\lor \big(\psi(y,x)\land \neg \psi(x,y)\big)\Big)$ (Antisymmetry).
\end{enumerate}

We now define a partial order $\prec$ on the set of all notation systems. Any tuples $\alpha=\langle \varphi,\psi,p,n\rangle$ and $\alpha'=\langle \varphi',\psi',p',n\rangle$ are $\prec$-incomparable if either $\varphi\ne\varphi'$, or $\psi\ne\psi'$, or $p\ne p'$. If $\alpha,\beta$ are of the form $\alpha=\langle\varphi,\psi,p,n\rangle$ and $\beta=\langle \varphi,\psi,p,m\rangle$, we put $\alpha \prec \beta$ if $\mathsf{Tr}_{\Sigma_1} \big( \psi(n,m) \big)$ but $\mathsf{Tr}_{\Sigma_1} \big(\neg \psi(m,n) \big)$. 

Clearly the relation $\prec$ and the property of being an ordinal notation system are expressible by $\Sigma_1$-formulas. In $\mathsf{EA}^+$ we could expand the language by a definable superexponentiation function $2^{x}_y$. Since the superexponentiation function is $\mathsf{EA}^+$ provably monotone, by a standard technique one could show that $\mathsf{EA}^+$ proves induction for the class $\Delta_0(2^x_y)$ of formulas with bounded quantifiers in the expanded language. It is easy to show that over $\mathsf{EA}^+$ the truth of $\Delta_0$-formulas according to the $\Sigma_1$-truth predicate could be expressed by a $\Delta_0(2^x_y)$ formula. Thus, the order $\prec$ and the property of being an ordinal notation system are expressible by $\Delta_0(2^x_y)$ formulas, which allows us to reason about them in $\mathsf{EA}^+$ is a straightforward manner.

Let us show that $\mathsf{EA}^+$ proves that $\prec$ is a disjoint union of linear orders. First we note that the theory $\mathsf{EA}^+$ proves the $\Pi_2$ soundness of $\mathsf{EA}$ (i.e.  $\mathsf{RFN}_{\Pi_2}(\mathsf{EA})$, see section below). And we note that for any $\alpha=\langle \varphi(x),\psi(x,y), p, n \rangle $ the conclusion of $p$ (conjunction of sentences (\ref{not_prop_1})--(\ref{not_prop_3})) is $\mathsf{EA}$-provably equivalent to a $\Pi_1$ sentence. Hence for any notation system $\alpha=\langle \varphi(x),\psi(x,y), p, n \rangle $ the theory $\mathsf{EA}^+$ proves that the corresponding conjunction of sentences (\ref{not_prop_1})--(\ref{not_prop_3}) is true. Using this we easily prove in $\mathsf{EA}^+$ that $\prec$ is a linear ordering, when restricted to the tuples that share the same first three components.

For an ordinal notation $\alpha$ the value of $|\alpha|$ is either an ordinal or $\infty$. If the lower cone $(\{\beta\mid \beta\prec \alpha\},\prec)$ is well-founded, then $|\alpha|$ is the ordinal isomorphic to the well-ordering $(\{\beta\mid \beta\prec \alpha\},\prec)$. Otherwise,  $|\alpha|=\infty$. In other words, $|\alpha|$ is the well-founded rank of $\alpha$ in the $\prec$-order.

An alternative (more standard) approach to treating ordinal notations in arithmetic is to fix an elementary ordinal notation up to some ordinal $\boldsymbol \alpha$. This is a fixed linear order $\mathbf{L}=(\mathcal{D}_{\mathbf{L}},\prec_{\mathbf{L}})$, where both $\mathcal{D}_{\mathbf{L}}\subseteq \mathbb{N}$ and $\prec_{\mathbf{L}} \subseteq \mathbb{N}\times \mathbb{N}$ are given by $\Delta_0$ formulas such that (i) $\mathbf{L}$ is provably linear in $\mathsf{EA}$, (ii) $\mathbf{L}$ is  well-founded, and (iii) the order type of $\mathbf{L}$ is $\boldsymbol \alpha$. It has been empirically observed that the ordinal notation systems that arise in ordinal analysis results in proof theory are of this kind; see, e.g., \cite{rathjen1999realm}. Note that from any $\mathbf{L}$ of this sort we could easily form an ordinal notation (in our sense) $\alpha$ such that there is a Kalmar elementary isomorphism $f$ between $\mathbf{L}$ and $(\{\beta\mid \beta\prec\alpha\},\prec)$; moreover, the latter is provable in $\mathsf{EA}^+$. 

Further we will work with ordinal notation systems that are given by some combinatorially defined system of terms and order on them. The standard example of such a system is the Cantor ordinal notation system up to $\varepsilon_0$. For the notations that we will consider it will be always possible to formalize in $\mathsf{EA}$ the definition and proof that the order is linear. Thus, as described above, we will be able to form an ordinal notation $\alpha$ such that there will be a natural isomorphism between $(\{\beta\mid \beta\prec\alpha\},\prec)$ and the initial combinatorially defined ordinal notation system. We will make transitions from combinatorial definitions of notation systems to ordinal notation systems in our sense without any further comments.

Moreover, we will use expressions like $\omega^{\alpha}$ and $\varepsilon_{\alpha}$, where $\alpha$ is some ordinal notation system. Let us consider a notation system $\alpha=\langle \varphi(x),\psi(x,y),p,n\rangle$ and define the notation system $\omega^{\alpha}=\langle \varphi'(x),\psi'(x,y),p',n'\rangle$. We want the order $\prec_{\omega^\alpha}$ to be the order on the terms $\omega^{a_1}+\ldots+\omega^{a_k}$, where $a_1\succeq_{\alpha}\ldots\succeq_{\alpha}a_k$. And the order $\prec_{\omega^\alpha}$ is defined as the usual order on Cantor normal forms, where we compare $a_i$ by the order $\prec_{\alpha}$. By arithmetizing this definition of $\prec_{\omega^{\alpha}}$ we get $\varphi'$, $\psi'$, and $p'$. We put $n'$ to be the number encoding the term $\omega^{n}$. Note that, according to this definition, $\alpha$ and $\omega^{\alpha}$ are $\prec$-incomparable. However, if $\alpha\prec \beta$, then $\omega^{\alpha}\prec \omega^{\beta}$.

The definition of the notation system $\varepsilon_{\alpha}$ is similar to that of $\omega^{\alpha}$. The system of terms for $\varepsilon_{\alpha}$ consists of nested Cantor normal forms built up from $0$ and elements $\varepsilon_{a}$, for $a\in \mathsf{dom}(\prec_{\alpha})$. The comparission of nested Cantor normal forms is defined in the standard fashion, where we compare elements $\varepsilon_{a}$ and $\varepsilon_b$ as $a\prec_{\alpha} b$.

\subsection{Reflection principles}

 Suppose $\mathcal{C}$ is some class of formulas in one of the languages that we consider and $T$ is a theory in the same language. The uniform $\mathcal{C}$ reflection principle $\mathsf{RFN}_\mathcal{C}(T)$ over $T$ is the schema
$$ \forall\vec{x} \big( \mathsf{Pr}_T(\varphi(\vec{x})) \rightarrow \varphi(\vec{x}) \big) $$
for all $\varphi\in\mathcal{C}$, where $\vec{x}$ are free number variables and $\varphi(\vec{x})$ contains no other variables.

In those cases for which we have a truth-definition for $\mathcal{C}$ in $T$ the scheme $\mathsf{RFN}_\mathcal{C}(T)$ can be axiomatized by the single sentence $$\forall \varphi \in \mathcal{C} \big (\mathsf{Pr}_T (\varphi) \rightarrow \mathsf{Tr}_{\mathcal{C}}(\varphi)\big).$$

Given an ordinal notation system $\prec $, we informally define the operation $\mathbf{R}_{\mathcal{C}}^{\cdot}(\cdot)$ of iterated $\mathcal{C}$ reflection along $\prec$ as follows.
\begin{flalign*}
\mathbf{R}^0_\mathcal{C}(T) &: = T \\
\mathbf{R}^\alpha_\mathcal{C}(T) &: = T + \bigcup_{\beta\prec\alpha}\mathsf{RFN}_\mathcal{C}\big(\mathbf{R}^\beta_\mathcal{C}(T)\big)\textrm{ for $\alpha\succ 0$.}
\end{flalign*}

More formally, we appeal to G\"{o}del's fixed point lemma in $\mathsf{EA}$.  We fix a formula $\mathsf{RFN}\mbox{-}\mathsf{Inst}_{\mathcal{C}}(U,x)$, where $U$ and $x$ are first-order variables, that formalizes the fact that $x$ is an instance of the scheme $\mathsf{RFN}_{\mathcal{C}}(U)$. We now want to define a $\Sigma_1$ formula $\mathsf{Ax}_{\mathbf{R}^\alpha_\mathcal{C}(T)}(x)$ (note that $\alpha$, $T$, and $x$ are arguments of the formula) that defines the set of axioms of the theories $\mathbf{R}^\alpha_\mathcal{C}(T)$. We define the formula as a fixed point:
$$\mathsf{EA}\vdash \mathsf{Ax}_{\mathbf{R}^\alpha_\mathcal{C}(T)}(x)\leftrightarrow  \big( \mathsf{Ax}_{T}(x) \vee \exists \beta \prec \alpha \;\mathsf{RFN}\mbox{-}\mathsf{Inst}_{\mathcal{C}}(\mathbf{R}^\beta_\mathcal{C}(T),x)) \big) ,$$
note that when we substitute $\mathbf{R}^\beta_\mathcal{C}(T)$ in $\mathsf{RFN}\mbox{-}\mathsf{Inst}_{\mathcal{C}}$ we actually substitute (the G\"{o}del number of) $\mathsf{Ax}_{\mathbf{R}^\beta_\mathcal{C}(T)}$.

Beklemishev introduced this approach to defining progressions of iterated reflection in \cite{beklemishev1995iterated}; the reader can find a more modern version of this approach in \cite{beklemishev2017reflection}. It is easy to prove that this definition of progressions of iterated reflection provides a unique (up to $\mathsf{EA}$ provable deductive equivalence) definition of the theories $\mathsf{RFN}^\alpha_\mathcal{C}(T)$. 

\subsection{Reflexive induction}

We often employ Schmerl's technique of \emph{reflexive induction}. Reflexive induction is a way of simulating large amounts of transfinite induction in weak theories. The technique is facilitated by the following theorem; we include the proof of the theorem, which is very short.

\begin{theorem}[Schmerl]\label{reflexive induction}
 Let $T$ be a recursively axiomatized theory (in one of the languages that we consider) that contains $\mathsf{EA}$. Suppose $$T \vdash \forall\alpha \Big(\mathsf{Pr}_{T} \big(\forall\beta \prec \alpha \; \varphi(\beta) \big) \rightarrow \varphi(\alpha)\Big).$$ Then $T \vdash \forall\alpha \;\varphi (\alpha)$.\footnote{Schmerl proved this result over the base theory $\mathsf{PRA}$. Beklemishev \cite{beklemishev2003proof} weakened the base theory to $\mathsf{EA}$.}
\end{theorem}

\begin{proof} Suppose that $T \vdash \forall\alpha \Big(\mathsf{Pr}_{T} \big(\forall\beta \prec \alpha \; \varphi(\beta) \big) \rightarrow \varphi(\alpha)\Big).$
We infer that 
$$T \vdash \forall\alpha \mathsf{Pr}_T\big(\forall \beta \prec\alpha\; \varphi(\beta)\big) \rightarrow \forall\alpha\;\varphi(\alpha),$$
whence it follows that
$$T\vdash  \mathsf{Pr}_T\big(\forall\alpha \;\varphi(\alpha)\big) \rightarrow\forall\alpha \;\varphi(\alpha).$$
L\"{o}b's theorem then yields $T\vdash\forall\alpha \;\varphi(\alpha)$.
\end{proof}

Accordingly, to prove claims of the form $T \vdash \forall\alpha \;\varphi (\alpha)$, we often prove that $T \vdash \forall\alpha \Big(\mathsf{Pr}_{T} \big(\forall\beta \prec \alpha\; \varphi(\beta) \big) \rightarrow \varphi(\alpha)\Big)$ and infer the desired claim by Schmerl's Theorem. While working inside $T$, we refer to the assumption $\mathsf{Pr}_{T} \big(\forall\beta \prec \alpha\; \varphi(\beta) \big)$ as the \emph{reflexive induction hypothesis}.


\section{Well-foundedness and reflection principles}
\label{dssection}

In this section we develop a technique for showing that certain orders on axiomatic theories exhibit a well-foundeness like properties. The coarsest order that we will consider is $\Pi^1_1$ reflection order for which we will prove that its restriction to $\Pi^1_1$ sound theories is well-founded. For weaker reflection and consistency orders we will prove only some well-foundedness like properties. Also we note that the same technique is used in \cite[Theorem~3.2]{enayat2019truth} to prove certain facts about axiomatic theories of truth and in \cite[Theorem~1.1]{lutz2019incompleteness} to prove a recursion-theoretic result concerning the hyper-degrees.

Our technique is inspired by H. Friedman's \cite{friedman1976uniformly} proof of the following result originally due to Steel \cite{steel1975descending}; recall that $\leq_T$ denotes Turing reducibility.

\begin{theorem}
Let $P\subset \mathbb{R}^2$ be arithmetic. Then there is no sequence $(x_n)_{n<\omega}$ of reals such that for every $n$, both $x_n \geq_T x'_{n+1}$ and also $x_{n+1}$ is the unique real $y$ such that $P(x_n,y)$.
\end{theorem}

Friedman and Steel were not directly investigating the well-foundedness of axiomatic systems, but rather an analogous phenomenon from recursion theory, namely, the well-foundedness of natural Turing degrees under Turing reducibility. The adaptability of Friedman's proof arguably strengthens the analogy between these phenomena.

In this section we study both first and second order theories. The first theory that we treat with our technique is $\mathsf{ACA}_0$, a subsystem of second-order arithmetic that has been widely studied in reverse mathematics. $\mathsf{ACA}_0$ is arithmetically conservative over $\mathsf{PA}$. We then turn to other applications of our technique. We consider $\mathsf{RCA}_0$, another subsystem of second-order arithmetic and familiar base theory from reverse mathematics. $\mathsf{RCA}_0$ is conservative over $\mathsf{I\Sigma}_1$. We then turn to first-order theories, and we study elementary arithmetic $\mathsf{EA}$ as our object theory.

\subsection{$\Pi^1_1$-Reflection}

In this subsection we examine the ordering $\prec_{\Pi^1_1}$ on r.e. extensions of $\mathsf{ACA}_0$, where
$$T\prec_{\Pi^1_1}U\stackrel{\mbox{\footnotesize \textrm{def}}}{\iff}  U\vdash \mathsf{RFN}_{\Pi^1_1}(T).$$
We will show that there are no infinite $\prec_{\Pi^1_1}$ descending sequences of $\Pi^1_1$ sound extensions of $\mathsf{ACA}_0$. We recall that, provably in $\mathsf{ACA}_0$, a theory $T$ is $\Pi^1_1$ sound if and only if $T$ is consistent with any true $\Sigma^1_1$ statement.

\begin{theorem}\label{well-foundedness_reflection}$(\mathsf{ACA}_0)$ The restriction of the order $\prec_{\Pi^1_1}$ to $\Pi^1_1$-sound r.e. extensions of $\mathsf{ACA}_0$ is well-founded. 
\end{theorem}
\begin{proof} In order to prove the result in $\mathsf{ACA}_0$ we show the inconsistency of the theory $\mathsf{ACA}_0$ plus the following statement $\mathsf{DS}$, which says that there \emph{is} a descending sequence of $\Pi^1_1$ sound extensions of $\mathsf{ACA}_0$ in the $\prec_{\Pi^1_1}$ ordering:
  $$\mathsf{DS}:=\exists E\colon \langle T_i\mid i\in \mathbb{N}\rangle (\mathsf{RFN}_{\Pi^1_1}(T_{0})\land \forall x \;\mathsf{Pr}_{T_x}(\mathsf{RFN}_{\Pi^1_1} (T_{x+1}))\wedge \forall x (T_x \sqsupseteq \mathsf{ACA}_0))$$
Note that $E\colon \langle T_i\mid i\in \mathbb{N}\rangle$ is understood to mean that $E$ is a set encoding a sequence $\langle T_0,T_1,T_2,\ldots \rangle$ of r.e. theories.
  
  If we prove that $\mathsf{ACA}_0+\mathsf{DS}$ proves its own consistency, then the inconsistency of $\mathsf{ACA}_0+\mathsf{DS}$ follows from G\"{o}del's second incompleteness theorem.  We reason in $\mathsf{ACA}_0+\mathsf{DS}$ to to prove consistency of $\mathsf{ACA}_0+\mathsf{DS}$.

  Let $E\colon \langle T_i\mid i\in \mathbb{N}\rangle $ be a sequence of theories witnessing the truth of $\mathsf{DS}$. Let us consider the sentence $\mathsf{F}$
  $$\exists U\colon  \langle S_i\mid i\in \mathbb{N}\rangle (S_0=T_1\land \forall x \;\mathsf{Pr}_{S_x}(\mathsf{RFN}_{\Pi^1_1}(S_{x+1})) \wedge\forall x (S_x \sqsupseteq \mathsf{ACA}_0)).$$ The sentence $\mathsf{F}$ is true since we could take $\langle T_{i+1}:i\in\mathbb{N}\rangle$ as $U$. It is easy to observe that $\mathsf{F}$ is $\Sigma^1_1$.
    
  From $\mathsf{RFN}_{\Pi^1_1}(T_{0})$ we get that $T_{0}$ is consistent with any true $\Sigma^1_1$ statement. Thus, we infer that $$\mathsf{Con}(T_{0} + \mathsf{F}).$$ Now using the fact that $\mathsf{Pr}_{T_0} \big(\mathsf{RFN}_{\Pi^1_1}(T_{1})\big)$ and that $T_0\sqsupseteq \mathsf{ACA}_0$ we conclude,
  $$\mathsf{Con}( \mathsf{ACA}_0 + \mathsf{RFN}_{\Pi^1_1}(T_{1}) + \mathsf{F}).$$
But it is easy to see that $\mathsf{RFN}_{\Pi^1_1}(T_{1}) + \mathsf{F}$ implies $\mathsf{DS}$ in $\mathsf{ACA}_0$. In particular, we may take $\langle T_1, T_2, ...\rangle$ as our new witness to $\mathsf{DS}$. Thus, we conclude that $\mathsf{Con}(\mathsf{ACA}_0+\mathsf{DS})$.  
\end{proof}

We now observe that a similar result holds over $\mathsf{RCA}_0$. To do so, we consider formulas from the complexity class $\Pi^1_1(\Pi^0_3)$ (see \textsection{2.4}). It is easy to see that the proof of Theorem \ref{RCA_0_reflection} remains valid if we replace the theory $\mathsf{ACA}_0$ with $\mathsf{RCA}_0$, the complexity class $\Pi^1_1$ with $\Pi^1_1(\Pi^0_3)$, and the complexity class $\Sigma^1_1$ with $\Sigma^1_1(\Pi^0_2)$. Thus, we also infer the following.

\begin{theorem} \label{RCA_0_reflection} 
$(\mathsf{RCA}_0)$ The restriction of the order $\prec_{\Pi^1_1(\Pi^0_3)}$ to $\Pi^1_1(\Pi^0_3)$-sound r.e. extensions of $\mathsf{RCA}_0$ theories is well-founded.
\end{theorem}

\subsection{$\Pi_3$ soundness}


In this subsection we study the complexity of descending sequences of r.e. theories with respect to $\Pi_3$ soundness. We recall that (provably in $\mathsf{EA}$) a theory $T$ is $\Pi_3$ sound just in case $T$ is 2-consistent, i.e., just in case $T$ is consistent with any true $\Pi_2$ sentence. 

\begin{theorem}
\label{2Con}
There is no recursively enumerable sequence $(T_n)_{n<\omega}$ of r.e.\ extensions of $\mathsf{EA}$ such that $T_0$ is $\Pi_3$ sound and such that for every $n$, $ T_n \vdash \mathsf{RFN}_{\Pi_3}(T_{n+1})$.
\end{theorem}

\begin{proof}
If the theorem fails, then the following sentence is true,
$$\mathsf{DS}:= \exists e\colon\langle T_i \mid i\in \mathbb{N}\rangle \; \Big(\mathsf{RFN}_{\Pi_3}(T_0) \wedge \forall x \;\mathsf{Pr}_{T_x} \big(\mathsf{RFN}_{\Pi_3}(T_{x+1}) \big) \Big)$$
where $\exists e:\langle T_i : i\in\mathbb{N}\rangle$ is understood to mean that $e$ is an index for a Turing machine enumerating the sequence $\langle T_0, T_1, ...\rangle$.

We show that $\mathsf{EA}+\mathsf{DS}$ proves its own consistency, whence, by G\"{o}del's second incompleteness theorem, $\mathsf{EA}+\mathsf{DS}$ is inconsistent and hence $\mathsf{DS}$ is false.

\textbf{Work in} $\mathsf{EA}+\mathsf{DS}$. Since $\mathsf{DS}$ is true, it has some witness $e\colon\langle T_i \mid i\in \mathbb{N}\rangle$. We now consider the sequence $e'$ that results from omitting $T_0$ from $e$. More formally, we consider the sequence $e'\colon\langle T'_i\mid i\in \mathbb{N}\rangle,$ which is numerated by the Turing functional $\{e'\}:x\mapsto \{e\}(x+1)$. That is, for each $i$, $T_i'=T_{i+1}$.

From $\mathsf{DS}$ we infer that for all $x$, $T_{x+1} \vdash \mathsf{RFN}_{\Pi_3}(T_{x+2})$. Thus, for every $x$, $T'_x \vdash \mathsf{RFN}_{\Pi_3}(T'_{x+1})$ by the definition of $e'$.

From the first conjunct of $\mathsf{DS}$ we infer that $\mathsf{RFN}_{\Pi_3}(T_0)$. That is, $T_0$ is consistent with any $\Pi_2$ truth. Thus, we infer that 
$$T_0+\forall x \;\mathsf{Pr}_{T'_x} \big(\mathsf{RFN}_{\Pi_3}(T'_{x+1})\big)$$ 
is consistent. 

On the other hand, from $\mathsf{DS}$ we infer that $T_0$ \emph{proves} the $\Pi_3$ soundness of $T'_0$. So it is consistent that $e'$ witnesses $\mathsf{DS}$.
\end{proof}

\subsection{Consistency}

In this subsection we provide a new proof of a theorem independently due to H. Friedman, Smorynski, and Solovay  (see \cite{lindstrom2017aspects,smorynski2012self}). Before stating the theorem we recall that, $\mathsf{EA}$ proves the equivalence of, the consistency sentences $\mathsf{Con}(T)$ and the $\Pi_1$-reflection principle $\mathsf{RFN}_{\Pi_1}(T)$.

\begin{theorem}
There is no recursively enumerable sequence $(T_n)_{n<\omega}$ of r.e.\ extensions of $\mathsf{EA}$ such that $T_0$ is consistent and such that $\mathsf{EA}\vdash \forall x \;\mathsf{Pr}_{T_x} \big(\mathsf{Con}(T_{x+1})\big)$.
\end{theorem}

\begin{proof}
Suppose, toward a contradiction, that there is a recursively enumerable sequence $(T_n)_{n<\omega}$ of r.e.\ extensions of $\mathsf{EA}$ such that $T_0$ is consistent and such that 
$$\mathsf{EA}\vdash \forall x \; \mathsf{Pr}_{T_x} \big( \mathsf{Con}(T_{x+1}) \big).$$ 
Since $\mathsf{EA}$ is sound, we also infer that for every $n$, $T_n \vdash \mathsf{Con}(T_{n+1})$.
Thus the following sentence is true.
$$\mathsf{DS}:= \exists e\colon\langle T_i \mid i\in \mathbb{N}\rangle \; \Big(\mathsf{Con}(T_0) \wedge \mathsf{Pr}_{\mathsf{EA}} \big( \forall x \; \mathsf{Pr}_{T_x} (\mathsf{Con}(T_{x+1})) \big) \wedge \forall x \;\mathsf{Pr}_{T_x} \big(\mathsf{Con}(T_{x+1}) \big) \Big)$$
where $\exists e:\langle T_i : i\in\mathbb{N}\rangle$ is understood to mean that $e$ is an index for a Turing machine enumerating the sequence $\langle T_0, T_1, ...\rangle$.

We show that $\mathsf{EA}+\mathsf{DS}$ proves its own consistency, whence, by G\"{o}del's second incompleteness theorem, $\mathsf{EA}+\mathsf{DS}$ is inconsistent and hence $\mathsf{DS}$ is false.

\textbf{Work in} $\mathsf{EA}+\mathsf{DS}$. Since $\mathsf{DS}$ is true, it has some witness $e\colon\langle T_i \mid i\in \mathbb{N}\rangle$. We consider the sequence $e'\colon\langle T'_i \mid i\in \mathbb{N}\rangle$ that results from dropping $T_0$ from the sequence produced by $e$. More formally, we consider the sequence $e'$ which is numerated by the Turing functional $\{e'\}:x\mapsto \{e\}(x+1)$.

\begin{claim}
$e'$ is provably a witness to $\mathsf{DS}$ in $T_0$. 
\end{claim}

To see that $e'$ provably witnesses the third conjunct of $\mathsf{DS}$ in $T_0$, we reason as follows.
\begin{flalign*}
\mathsf{EA}\vdash&\forall x \; \mathsf{Pr}_{T_{x+1}} \mathsf{Con}(T_{x+2}) \textrm{ by $\mathsf{DS}$.}\\
\mathsf{EA}\vdash&\forall x \; \mathsf{Pr}_{T'_x} \mathsf{Con}(T'_{x+1}) \textrm{ since $T'_x=T_{x+1}$ by definition of $e'$.} \\
T_0\vdash&\forall x \;\mathsf{Pr}_{T'_x} \mathsf{Con}(T'_{x+1}) \textrm{ since $T_0$ extends $\mathsf{EA}$.}
\end{flalign*}

To see that $e'$ provably witnesses the second conjunct of $\mathsf{DS}$ in $T_0$, we reason as follows.
\begin{flalign*}
\mathsf{EA}\vdash&\forall x \;\mathsf{Pr}_{T'_x} \mathsf{Con}(T'_{x+1}) \textrm{ as above.}\\
\mathsf{EA}\vdash&\mathsf{Pr}_\mathsf{EA} \big(\forall x \;\mathsf{Pr}_{T'_x} \mathsf{Con}(T'_{x+1})\big) \textrm{ by the $\Sigma_1$ completeness of $\mathsf{EA}$.}\\
T_0\vdash&\mathsf{Pr}_\mathsf{EA} \big(\forall x \;\mathsf{Pr}_{T'_x} \mathsf{Con}(T'_{x+1})\big) \textrm{ since $T_0$ extends $\mathsf{EA}$.}
\end{flalign*}

We now show that $e'$ provably witnesses the first conjunct of $\mathsf{DS}$ in $T_0$. From the first conjunct of $\mathsf{DS}$ we infer that $\mathsf{Con}(T_0)$. It follows that $T_0$ is $\Pi_1$ sound. We reason as follows.
\begin{flalign*}
&T_{0} \vdash \mathsf{Con}(T_{1}) \textrm{ by $\mathsf{DS}$.}\\
&T_0 \vdash \mathsf{Con}(T'_0) \textrm{ since provably $T'_0=T_1$.}
\end{flalign*}
We then infer that $\mathsf{Con}(T'_0)$ by the $\Pi_1$ soundness of $T_0$. So $e'$ is provably a witness to $\mathsf{DS}$ in a consistent theory. Therefore $\mathsf{EA}+\mathsf{DS}$ is consistent.
\end{proof}

\begin{remark}
Note that we just proved the non-existence of $\mathsf{EA}$-\emph{provably} descending r.e. sequences. Without the condition of $\mathsf{EA}$ provability such descending sequences \emph{do} exist. H. Friedman, Smorynski, and Solovay independently proved that there is a recursive sequence $\langle T_0, T_1, ...\rangle$ of consistent extensions of $\mathsf{EA}$ such that for all $n$, $T_n \vdash \mathsf{Con}(T_{n+1})$, answering a question of Gaifman; see \cite{smorynski2012self} for details.
\end{remark}

\subsection{$\Pi_2$ soundness}

We now know that there are \emph{no} recursive descending sequences of $\Pi_3$ sound theories with respect to the $\Pi_3$ reflection order, but there \emph{are} recursive descending sequences of consistent theories with respect to consistency strength. In this subsection we treat the remaining case, namely, $\Pi_2$ soundness. We prove that there \emph{is} an infinite sequences $\langle T_0, T_1, ...\rangle$ of $\Pi_2$ sound extensions of $\mathsf{EA}$ such that for all $n$, $T_n \vdash \mathsf{RFN}_{\Pi_2}(T_{n+1})$. In this sense, Theorem \ref{2Con} is best possible.


In the section, for technical reasons it will be useful for us to impose some natural conditions on our proof predicate.
We make sure that any proof in our proof system has only one conclusion, whence $$\mathsf{EA}\vdash \forall x,y_1,y_2 \Big( \big(\mathsf{Prf}_{T}(x,y_1)\land \mathsf{Prf}_{T}(x,y_2) \big) \to y_1=y_2 \Big).$$ Moreover, we arrange the proof system so that indices for statements are less than or equal to the indices for their proofs, i.e.,
\begin{equation}\label{proof_greater_than_theorem}\mathsf{EA}\vdash\forall x,y(\mathsf{Prf}_{T}(x,y)\to y\le x).\end{equation}
Note that the conclusions of the theorems in our paper are not sensitive to the choice of proof predicate as long as the resulting provability predicates are $\mathsf{EA}$-provably equivalent. And it is easy to see that even if our initial choice of $\mathsf{Prf}_{T}(x,\varphi)$ didn't satisfied the mentioned conditions, it is easy to modify it to satisfy the conditions, while preserving the provability predicate $\mathsf{Pr}_{T}(\varphi)$ up to $\mathsf{EA}$-provable equivalence.

Before proving the theorem we make a few more remarks preliminary remarks. We use the symbol $\dotdiv$ to denote the truncated subtraction function, i.e., $n\dotdiv m = n-m$ if $n>m$ and $0$ otherwise. We remind the reader that, provably in $\mathsf{EA}$, a theory is $\Sigma_1$ sound if and only if it is $\Pi_2$ sound. We also pause to make the following remark, which will invoke in the proof of the theorem.

\begin{remark}\label{HBL}
For any $\Pi_2$ sound extension $T$ of $\mathsf{EA}$, the theory $T+ \neg \mathsf{RFN}_{\Pi_2}(T)$ is $\Pi_2$ sound. This is actually an instance G\"odel's second incompleteness theorem that is applied to $1$-provability rather than the ordinary provability. Recall that $1$-provability predicate $1\mbox{-}\mathsf{Pr}_T(\varphi)$ for a theory $T$ is
\begin{equation}
  \exists \psi\in \Sigma_2\big(\mathsf{Tr}_{\Sigma_2}(\psi)\land \mathsf{Pr}_T(\psi\to\varphi)\big).
\end{equation}
The consistency notion that corresponds to $1$-provability is precisely $\Pi_2$-soundness:
\begin{equation}
  \mathsf{EA}\vdash \forall \varphi\;\big(\lnot 1\mbox{-}\mathsf{Pr}_T(\lnot \varphi)\mathrel{\leftrightarrow}\mathsf{RFN}_{\Pi_2}(T+\varphi)\big).
\end{equation}
It is easy to see that $1$-provability predicate for a theory $T$ satisfies the usual Hilbrt-Bernays-L\"ob derivability conditions. Thus G\"odel's second incompleteness theorem for it states that if a theory $T\sqsupseteq \mathsf{EA}$ is $\Pi_2$-sound, then $\mathsf{RFN}_{\Pi_2}(T)$ is not $1$-provable in $T$. And the latter is equivalent to $\Pi_2$-soundness of $T+\lnot\mathsf{RFN}_{\Pi_2}(T)$.
\end{remark}

We are now ready for the proof of the theorem.

\begin{theorem} \label{Pi_2_descending} There is a recursive sequence $(\varphi_n)_{n<\omega}$ of $\Pi_2$-sound sentences such that, for each $n$, $\mathsf{EA}+\varphi_n\vdash \Ref{\Pi_2}{\mathsf{EA}}(\mathsf{EA}+\varphi_{n+1})$.
\end{theorem}

\begin{proof}
For each $n\in\mathbb{N}$, we define the sentence $\varphi_n$ as follows:
$$ \varphi_n : = \exists \psi \in \Sigma_1 \exists p \Big( \mathsf{Prf}_{\mathsf{I}\Sigma_2}(p,\psi) \wedge \neg \mathsf{True}_{\Sigma_1}(\psi) \wedge \mathsf{RFN}_{\Pi_2}\big(\mathbf{R}_{\Pi_2}^{p\dotdiv n}(\mathsf{EA})\big) \Big)$$
That is, $\varphi_n$ expresses ``$\mathsf{I}\Sigma_2$ proves a false $\Sigma_1$ sentence via a proof $p$, and $\Pi_2$ reflection for $\mathsf{EA}$ can be iterated up to $p\dotdiv n$.''

The motivation for picking that individual formula is as follows: To find a descending sequence, we will iterate $\Pi_2$ reflection up to some non-standard number. So we need to make sure that our formula forces a certain number to be non-standard but without implying any false $\Pi_2$ sentences. The way we do that is by saying that $\mathsf{I}\Sigma_2$ proves a false $\Sigma_1$ sentence. This has (we will show) no false $\Pi_2$ consequences. However, (the code of) any proof witnessing a failure of $\Sigma_1$ soundness in $\mathsf{I}\Sigma_2$ must be non-standard. We find our descending sequence by iterating $\Pi_2$ reflection up to this non-standard number.

Now the formal details start. We need to check that $\varphi_n$ is $\Pi_2$ sound for each $n$, and that $\mathsf{EA}+ \varphi_n \vdash \mathsf{RFN}_{\Pi_2}(\mathsf{EA}+\varphi_{n_1})$.

\begin{claim}
$\mathsf{EA} + \varphi_n$ is $\Pi_2$ sound for each $n$.
\end{claim}

The first thing to note is that
\begin{equation}
\mathsf{I}\Sigma_2 \vdash \forall x \;\mathsf{RFN}_{\Pi_2}\big(\mathbf{R}_{\Pi_2}^x(\mathsf{EA})\big),
\end{equation}
where the $\mathsf{I}\Sigma_2$-proof is the induction on $x$. Recall that $\mathsf{I}\Sigma_2 \equiv I\Pi_2$ and  $\forall x\; \mathsf{RFN}_{\Pi_2}\big(\mathbf{R}_{\Pi_2}^x(\mathsf{EA})\big)$ is a $\Pi_2$-formula, hence $\mathsf{I}\Sigma_2$ could formalize the necessary induction. Also it is known that $\mathsf{I}\Sigma_2\sqsupseteq\mathsf{I}\Sigma_1\equiv \mathsf{EA}+\mathsf{RFN}_{\Pi_3}(\mathsf{EA})$ and that
$$\mathsf{EA}+\mathsf{RFN}_{\Pi_3}(\mathsf{EA})\vdash \psi\to \mathsf{RFN}_{\Pi_2}(\mathsf{EA}+\psi),$$ for any $\Pi_2$-formula $\psi$. This allows us to verify the base and step of the induction in $\mathsf{I}\Sigma_2$.

The second thing to note is that, since $\Pi_2$ reflection is provably equivalent (in $\mathsf{EA}$) to $\Sigma_1$ reflection, it follows that: 
\begin{equation}
\mathsf{I}\Sigma_2 + \neg \mathsf{RFN}_{\Pi_2}(\mathsf{I}\Sigma_2) \vdash \exists \psi \in \Sigma_1 \exists p \big( \mathsf{Prf}_{\mathsf{I}\Sigma_2}(p,\psi) \wedge \neg \mathsf{True}_{\Sigma_1}(\psi) \big)
\end{equation}

Putting these two observations together, we infer that, for each standard $n\in\mathbb{N}$,
\begin{equation}
\mathsf{I}\Sigma_2 + \neg \mathsf{RFN}_{\Pi_2}(\mathsf{I}\Sigma_2) \vdash \exists \psi \in \Sigma_1 \exists p \big( \mathsf{Prf}_{\mathsf{I}\Sigma_2}(p,\psi) \wedge \neg \mathsf{True}_{\Sigma_1}(\psi) \wedge \mathsf{RFN}_{\Pi_2}^{p\dotdiv n}(\mathsf{EA}) \big)
\end{equation}
which is just to say that for each standard $n\in\mathbb{N}$, $\mathsf{I}\Sigma_2 + \neg \mathsf{RFN}_{\Pi_2}(\mathsf{I}\Sigma_2) \vdash \varphi_n$. Thus, to see that $\mathsf{EA}+\varphi_n$ is $\Pi_2$ sound, it suffices to observe that $\mathsf{I}\Sigma_2 + \neg \mathsf{RFN}_{\Pi_2}(\mathsf{I}\Sigma_2)$ is $\Pi_2$ sound. The latter claim follows immediately from Remark \ref{HBL}.

Before checking that $\mathsf{EA}+ \varphi_n \vdash \mathsf{RFN}_{\Pi_2}(\mathsf{EA}+\varphi_{n+1})$, we will establish the following lemma:

\begin{lemma}\label{standard}
For all standard $n\in\mathbb{N}$, $$\mathsf{EA} \vdash \forall p \forall \psi\in\Sigma_1 \Big( \big( \mathsf{Prf}_{\mathsf{I}\Sigma_2}(p,\psi) \wedge \neg \mathsf{True}_{\Sigma_1}(\psi) \big) \rightarrow p>n  \Big).$$
\end{lemma}

\begin{proof}
The first thing to note is that (by the $\Sigma_1$ soundness of $\mathsf{I}\Sigma_2$ and the $\Sigma_1$ completeness of $\mathsf{EA}$) for any $\psi\in\Sigma_1$, if $\mathsf{I}\Sigma_2\vdash\psi$ then also $\mathsf{EA}\vdash \psi$. Now, for any standard $p\in\mathbb{N}$, $\mathsf{EA}$ can check whether $p$ constitutes an $\mathsf{I}\Sigma_2$ proof of a $\Sigma_1$ sentence $\psi$, and if $p$ does constitute such a proof, then $\mathsf{EA}$ will prove $\psi$ as well. That is, for each standard $p\in\mathbb{N}$:
$$\mathsf{EA} \vdash \forall \psi \in \Sigma_1 \Big( \mathsf{Prf}_{\mathsf{I}\Sigma_2}(p,\psi) \rightarrow \mathsf{True}_{\Sigma_1}(\psi) \Big)$$
It follows that for each standard $n\in\mathbb{N}$:
$$\mathsf{EA} \vdash \forall p\leq n \forall \psi \in \Sigma_1 \Big( \mathsf{Prf}_{\mathsf{I}\Sigma_2}(p,\psi) \rightarrow \mathsf{True}_{\Sigma_1}(\psi) \Big)$$
Whence for each standard $n\in\mathbb{N}$:
$$\mathsf{EA} \vdash \forall p \forall \psi \in \Sigma_1 \Big(  \big(\mathsf{Prf}_{\mathsf{I}\Sigma_2}(p,\psi) \wedge \neg \mathsf{True}_{\Sigma_1}(\psi) \big) \rightarrow p>n \Big)$$
This completes the proof of the lemma.
\end{proof}

With the lemma on board, we are now ready to verify the following claim:

\begin{claim}
For each $n\in\mathbb{N}$, $$\mathsf{EA}+ \varphi_n \vdash \mathsf{RFN}_{\Pi_2}(\mathsf{EA}+\varphi_{n+1}).$$
\end{claim}

Let's fix an $n\in\mathbb{N}$ and \textbf{reason in $\mathsf{EA}+\varphi_n$:}

According to $\varphi_n$, there is an $\mathsf{I}\Sigma_2$ proof $p$ of a false $\Sigma_1$ sentence $\psi$ and  $\mathsf{RFN}_{\Pi_2}(\mathbf{R}_{\Pi_2}^{p\dotdiv n} (\mathsf{EA}))$ is $\Pi_2$-sound. From Lemma \ref{standard} we infer that $p>n$. It follows that $p\dotdiv n >0$, whence  $p\dotdiv n = (p\dotdiv (n+1))+1$. Hence
\begin{equation}
  \mathbf{R}_{\Pi_2}^{p\dotdiv n} (\mathsf{EA})\equiv \mathbf{R}_{\Pi_2}^{(p\dotdiv (n+1))+1} (\mathsf{EA})\equiv \mathsf{EA}+\mathsf{RFN}_{\Pi_2} \big( \mathbf{R}_{\Pi_2}^{p\dotdiv (n+1)} (\mathsf{EA}) \big).
\end{equation}
Thus
$$  \mathsf{RFN}_{\Pi_2}\Big(\mathsf{EA}+\mathsf{RFN}_{\Pi_2} \big( \mathbf{R}_{\Pi_2}^{p\dotdiv (n+1)} (\mathsf{EA}) \big)\Big)$$
Since $\mathsf{Prf}_{\mathsf{I}\Sigma_2}(p,\psi)$ is a true $\Sigma_1$ sentence and $\psi$ is a false $\Sigma_1$ sentence we infer that
$$\mathsf{RFN}_{\Pi_2} \Big( \mathsf{EA} + \mathsf{Prf}_{\mathsf{I}\Sigma_2}(p,\psi) + \neg\mathsf{True}_{\Sigma_1}(\psi) + \mathsf{RFN}_{\Pi_2} \big( \mathbf{R}_{\Pi_2}^{p\dotdiv (n+1)} (\mathsf{EA}) \big) \Big)$$
Which straightforwardly implies $\mathsf{RFN}_{\Pi_2} ( \mathsf{EA} + \varphi_{n+1})$. This completes the proof of the theorem.
\end{proof}

\begin{question} In Theorem \ref{2Con} and Theorem \ref{Pi_2_descending} we studied how strong reflection principles should be to guarantee that there are no \emph{recursive} descending sequences in the corresponding reflection order. It is natural to ask how this result could be generalized to higher Turing degrees.

  Let $n$ be a natural number. For which $m$ is there a sequence $\langle T_i \mid i\in \mathbb{N}\rangle$ recursive in $0^{(n)}$ such that all $T_i$ are $\Pi_m$ sound extensions of $\mathsf{EA}$ and $T_i\vdash \mathsf{RFN}_{\Pi_m}(T_{i+1})$, for all $i$? The same question for $\Sigma_m$?
\end{question}


\section{Iterated reflection and conservation}
\label{conservativity_section}

In this section we prove a number of conservation theorems relating iterated reflection and transfinite induction.  These results are inspired by the following theorem, which is often known as \emph{Schmerl's formula} \cite{schmerl1979fine}. For an ordinal notation system $\alpha$, $\omega_n^{\alpha}$ is the result of $n$-applications of $\omega$-exponentiation (see  \textsection{\ref{ordinal_notations}}), starting with $\alpha$, i.e., $\omega_0^{\alpha}=\alpha$ and $\omega_{n+1}^\alpha=\omega^{\omega_n^\alpha}$.

\begin{theorem}[Schmerl]
\label{original Schmerl}  Let $n,m$ be natural numbers. In $\mathsf{EA}^+$, for any notation system $\alpha$, $$\mathbf{R}^\alpha_{\Pi^0_{n+m}}(\mathsf{EA}^+) \equiv_{\Pi^0_n} \mathbf{R}^{\omega_m(\alpha)}_{\Pi^0_n}(\mathsf{EA}^+).$$
\end{theorem}


Schmerl's formula is a useful tool for calculating the proof-theoretic ordinals of first-order theories. In this section we will develop tools in the mold of Schmerl's formula for calculating the proof-theoretic ordinals of second-order theories. Throughout this section we will rely on the following analogue of Theorem \ref{original Schmerl} that is also due to Schmerl \cite{schmerl1982iterated}. 

\begin{theorem}[Schmerl]
\label{generalized Schmerl} Provably in $\mathsf{EA}^+$, for any ordinal notation $\alpha$, $$\mathbf{R}^\alpha_{\mathbf{\Pi}^0_\infty}(\mathsf{PA}(X)) \equiv_{\mathbf{\Pi}^0_n} \mathbf{R}^{\varepsilon_\alpha}_{\mathbf{\Pi}^0_n}(\mathsf{EA}^+(X)).$$
\end{theorem}

Note that the versions of Schmerl's formulas that we give above aren't exactly what Schmerl proved, but rather versions of the formulas that are natural given the notation of our paper. And they could be proved by either application of Schmerl's technique or Beklemishev's technique \cite{beklemishev2003proof}. In fact in a early preprint of this paper \cite[\textsection6.2]{pakhomov2018reflection_v1} we provided a proof of Theorem \ref{generalized Schmerl}, however since the technique that we used wasn't new and the result is just a slight variation of \cite{schmerl1982iterated} we removed it from the paper.

Here is a roadmap for the rest of this section. In \textsection{\ref{RCA_section}} we prove Theorem \ref{rca theorem} that states that $$\mathbf{R}^\alpha_{\Pi^1_1(\Pi^0_3)}(\mathsf{RCA}_0) \equiv_{\boldsymbol\Pi^0_\infty} \mathbf{R}^{1+\alpha}_{\mathbf{\Pi}^0_3}(\mathsf{EA}^+(X)).$$ In \textsection{4.3} we use this result to prove Theorem \ref{main tool}, i.e., that $$\mathbf{R}^\alpha_{\Pi^1_1}(\mathsf{ACA}_0)\equiv_{\Pi^1_1(\Pi^0_3)} \mathbf{R}^{\varepsilon_\alpha}_{\Pi^1_1(\Pi^0_3)}(\mathsf{RCA}_0).$$

In \textsection{\ref{proof-theoretic_ordinal}} we will combine Theorem \ref{main tool} with the results from \textsection{\ref{dssection}} (especially Theorem \ref{well-foundedness_reflection} and Theorem \ref{RCA_0_reflection}) to establish connections between iterated reflection and ordinal analysis. In particular, we will use iterated reflection principles to calculate the proof-theoretic ordinals of a wide range of theories.

Before continuing, we alert the reader that many of the proofs in this section use Schmerl's technique of reflexive induction. For a description of this technique, please see \textsection{2.4}.
 
\subsection{Iterated reflection and recursive comprehension}
\label{RCA_section}

Recall that there are no descending chains in the $\mathsf{RFN}_{\Pi^1_1(\Pi^0_3)}$ ordering of ${\Pi^1_1(\Pi^0_3)}$ sound extensions of $\mathsf{RCA}_0$ (this is Theorem \ref{RCA_0_reflection}). In this subsection we investigate iterated $\Pi^1_1(\Pi^0_3)$ reflection over the theory $\mathsf{RCA}_0$. The main result of this subsection is that $\mathbf{R}^\alpha_{\Pi^1_1(\Pi^0_3)}(\mathsf{RCA}_0)$ is $\Pi^1_1$ conservative over $\mathbf{R}^{1+\alpha}_{\mathbf{\Pi}^0_3}(\mathsf{EA}^+(X))$. This result will be used in the next section to calculate proof-theoretic ordinals of subsystems of second-order arithmetic.

Before proving the theorem we prove a few lemmas. These lemmas concern proof-theoretic properties of theories that are closed under an inference rule that we call the $\mathbf{\Delta}^0_1$ \emph{substitution rule}.

\begin{definition} Suppose  $\varphi$ and $\theta(x)$ are  $\boldsymbol \Pi^0_{\infty}$ formulas that may have other free variables. We denote by $\varphi[\theta(x)]$ the result of substituting the formula $\theta(x)$ in for the free set variable $X$, i.e. to obtain $\varphi[\theta(x)]$ we first rename all the bounded variables of $\varphi$ in order to ensure that there are no clashes with free variables of $\theta$ and then replace each atomic subformula of $\varphi$ of the form $t\in X$ with $\theta(t)$.
\end{definition}

\begin{definition}  
 We write $\mathsf{Subst}_{\mathbf{\Delta}^0_1}[\varphi]$ to denote the formula
$$\forall \theta_1(x)\forall\theta_2(x)\Big(\forall y \big(\mathsf{Tr}_{\mathbf{\Pi}^0_1}(\theta_1(y))\leftrightarrow\mathsf{Tr}_{\mathbf{\Sigma}^0_1}(\theta_2(y))\big) \rightarrow \varphi[\mathsf{Tr}_{\mathbf{\Pi}^0_1}(\theta_1(x))]\Big).$$

A theory $T$ is closed under the $\mathbf{\Delta}^0_1$ \emph{substitution rule} if, for any formula $\psi(X)$, whenever $T\vdash \psi(X)$ then $T\vdash \mathsf{Subst}_{\mathbf{\Delta}^0_1}[\psi]$.
\end{definition}

Recall that there is a translation $\varphi(X)\longmapsto \forall X\;\varphi(X)$ from the set of $\boldsymbol\Pi^0_{\infty}$ sentences to the set of sentences of the language of second order arithmetic. Recall also that we are regarding the pseudo-$\Pi^1_1$ language as a sublanguage of the language of second order arithmetic by identifying each pseudo $\Pi^1_1$ sentence with its translation.

\begin{lemma}
\label{closure 2}
$(\mathsf{EA}^+)$ For each $\boldsymbol \Pi^0_{\infty}$ sentence $\varphi(X)$ the following are equivalent.
\begin{enumerate}
\item $\mathsf{RCA}_0 + \forall X\;\varphi(X)$ is $\mathbf{\Pi}^0_\infty$ conservative over $\mathsf{I\Sigma}_1(X)+\varphi(X)$.
\item $\mathsf{I\Sigma}_1(X)+\varphi(X)$ is closed under the $\mathbf{\Delta}^0_1$ substitution rule.
\item $\mathsf{I\Sigma}_1(X)+\varphi(X)$ proves $\mathsf{Subst}_{\mathbf{\Delta}^0_1}[\varphi]$.
\end{enumerate}
\end{lemma}

\begin{proof}
We work in $\mathsf{EA}^+$ and consider a  $\boldsymbol \Pi^0_{\infty}$ sentence $\varphi(X)$.

$(1)\rightarrow(2)$: Suppose that $\mathsf{RCA}_0 + \forall X\;\varphi(X)$ is $\Pi^1_1$ conservative over $\mathsf{I\Sigma}_1(X)+\varphi(X)$. Suppose that $\mathsf{I\Sigma}_1(X)+\varphi(X) \vdash \psi(X)$. Then $\mathsf{RCA}_0 + \forall X\;\varphi(X) \vdash \psi(X)$. Applying recursive comprehension, we derive $\mathsf{RCA}_0 + \forall X\;\varphi(X) \vdash \mathsf{Subst}_{\mathbf{\Delta}^0_1}[\psi]$. Hence, by $\Pi^1_1$ conservativity, $\mathsf{I\Sigma}_1(X)+\varphi(X) \vdash \mathsf{Subst}_{\mathbf{\Delta}^0_1}[\psi]$.

$(2)\rightarrow(3)$: By application of the $\mathbf{\Delta}^0_1$ substitution rule to $\varphi$.

$(3)\rightarrow(1)$: Suppose that $\mathsf{I\Sigma}_1(X)+\varphi(X)$ proves $\mathsf{Subst}_{\mathbf{\Delta}^0_1}[\varphi]$. We recall the well-known $\omega$-interpretation of $\mathsf{RCA}_0$ into $\mathsf{I\Sigma_1}(X)$ wherein we interpret sets by indices for $X$-recursive sets; see, e.g., \cite[\textsection IX.1]{simpson2009subsystems}. The image of the sentence $\forall X\;\varphi(X)$ under this interpretation is the sentence $\mathsf{Subst}_{\mathbf{\Delta}^0_1}[\varphi]$. This latter sentence is provable in $\mathsf{I\Sigma}_1(X)+\varphi(X)$ by assumption. Thus,  this interpretation actually interprets $\mathsf{RCA}_0 + \forall X\;\varphi(X)$ in  $\mathsf{I\Sigma}_1(X)+\varphi(X)$. Therefore, for any sentence $\psi(X)$, if $\mathsf{RCA}_0 + \forall X\;\varphi(X)$ proves $\forall X\;\psi(X)$, then $\mathsf{I\Sigma}_1(X)+\varphi(X)$ proves $\mathsf{Subst}_{\mathbf{\Delta}^0_1}[\psi]$, which is the image of $\forall X\;\psi(X)$ under the interpretation. Obviously, $\mathsf{I\Sigma}_1(X)+\varphi(X)\vdash \mathsf{Subst}_{\mathbf{\Delta}^0_1}[\psi]\to \psi(X)$, for any $\mathbf{\Pi}^0_{\infty}$ formula $\psi(X)$. Therefore,  $\mathsf{RCA}_0 + \forall X\;\varphi(X)$ is $\mathbf{\Pi}^0_\infty$ conservative over $\mathsf{I\Sigma}_1(X)+\varphi(X)$. 
\end{proof}

\begin{question} Combining Theorem \ref{RCA_0_reflection} and Lemma \ref{closure 2} it is easy to observe that the restriction of the order $\prec_{\boldsymbol\Pi^0_3}$ to $\boldsymbol \Pi^0_3$-sound r.e.\  extensions of $\mathsf{I\Sigma}_1(X)$ that are closed under the $\boldsymbol\Delta^0_1$-substitution rule is well-founded. Could we drop the condition on closure under the $\boldsymbol\Delta^0_1$-substitution rule? For which $n$ is the restriction of the order $\prec_{\boldsymbol\Pi^0_n}$ to $\boldsymbol \Pi^0_n$-sound r.e.  extensions of $\mathsf{I\Sigma}_1(X)$ well-founded? 
\end{question}

\begin{remark}\label{sigma1equivalence}
We recall that $\mathsf{I\Sigma_1}\equiv\mathsf{EA}^+ + \mathsf{RFN}_{\Pi_3}(\mathsf{EA}^+)$. See, e.g., \cite{beklemishev2004provability}. The same argument could be used to show that $\mathsf{I\Sigma_1}(X)\equiv\mathsf{EA}^+(X) + \mathsf{RFN}_{\boldsymbol\Pi^0_3}(\mathsf{EA}^+(X))$.
\end{remark}

\begin{lemma}
\label{closure 1}
$(\mathsf{EA}^+)$ If $\mathsf{EA}^+$ proves ``$T\sqsupseteq\mathsf{I\Sigma}_1(X)$ is closed under the $\mathbf{\Delta}^0_1$ substitution rule,'' then $\mathsf{EA}^+(X)+\mathsf{RFN}_{\mathbf{\Pi^0_3}}(T)$ is closed under the $\mathbf{\Delta}^0_1$ substitution rule.
\end{lemma}

\begin{proof}
Suppose that $\mathsf{EA}^+$ proves ``$T\sqsupseteq\mathsf{I\Sigma}_1(X)$ is closed under the $\mathbf{\Delta}^0_1$ substitution rule.'' Let us use the name $U$ for the theory $\mathsf{EA}^+(X) + \mathsf{RFN}_{\mathbf{\Pi}^0_3}(T)$. We want to show that $U$ is closed under the $\mathbf{\Delta}^0_1$ substitution rule. Note that, by Remark \ref{sigma1equivalence}, $U$ contains $\mathsf{I\Sigma_1}(X)$. That is, $U\equiv\mathsf{I\Sigma_1}(X)+\mathsf{RFN}_{\mathbf{\Pi}^0_3}(T)$. Over $\mathsf{EA}(X)$, the reflection schema $\mathsf{RFN}_{\boldsymbol\Pi^0_3}(T)$ is equivalent to  
$$\forall \varphi\in \mathbf{\Pi}^0_3 \Big(\mathsf{Pr}_T\big(\mathsf{Tr}_{\mathbf{\Pi}^0_3}(\varphi)\big)\rightarrow\mathsf{Tr}_{\mathbf{\Pi}^0_3}(\varphi)\Big).$$
Thus, by Lemma \ref{closure 2}, it suffices to show that $U$ proves  
$$\mathsf{Subst}_{\mathbf{\Delta}^0_1}[\forall \varphi\in \mathbf{\Pi}^0_3 \Big(\mathsf{Pr}_T\big(\mathsf{Tr}_{\mathbf{\Pi}^0_3}(\varphi)\big)\rightarrow\mathsf{Tr}_{\mathbf{\Pi}^0_3}(\varphi)\Big)].$$
But since the formula $\mathsf{Pr}_T(\mathsf{Tr}_{\mathbf{\Pi}^0_3}(\varphi))$  doesn't contain occurences of $X$, we could push $\mathsf{Subst}_{\mathbf{\Delta}^0_1}$ under the quantifier, i.e., it will be sufficient to show that 
 $$U \vdash \forall \varphi\in \mathbf{\Pi}^0_3 \Big(\mathsf{Pr}_T\big(\mathsf{Tr}_{\mathbf{\Pi}^0_3}(\varphi)\big)\rightarrow\mathsf{Subst}_{\mathbf{\Delta}^0_1}[\mathsf{Tr}_{\mathbf{\Pi}^0_3}(\varphi)]\Big).$$

Observe that $\mathsf{Subst}_{\mathbf{\Delta}^0_1}[\mathsf{Tr}_{\mathbf{\Pi}^0_3}(\varphi)]$ is equivalent to a $\mathbf{\Pi}^0_3$ formula over $\mathsf{EA}(X)$. We reason as follows.
\begin{flalign*}
U &\vdash \textrm{ ``$T$ is closed under the $\mathbf{\Delta}^0_1$ substitution rule,'' by assumption.}\\
U &\vdash \forall\varphi\in \mathbf{\Pi}^0_3\Big(\mathsf{Pr}_T\big(\mathsf{Tr}_{\mathbf{\Pi}^0_3}(\varphi)\big)\rightarrow\mathsf{Pr}_T(  \mathsf{Subst}_{\mathbf{\Delta}^0_1}[\mathsf{Tr}_{\mathbf{\Pi}^0_3}(\varphi)])\Big) \\
U &\vdash \forall \varphi\in \mathbf{\Pi}^0_3 \Big(\mathsf{Pr}_T\big(\mathsf{Tr}_{\mathbf{\Pi}^0_3}(\varphi)\big)\rightarrow \mathsf{Subst}_{\mathbf{\Delta}^0_1}[\mathsf{Tr}_{\mathbf{\Pi}^0_3}(\varphi)]\Big) \textrm{ by $\mathsf{RFN}_{\mathbf{\Pi}^0_3}(T)$}. 
\end{flalign*}
This concludes the proof of the lemma.
\end{proof}

With these lemmas on board we are ready for the proof of the main theorem of this subsection.

\begin{theorem}
\label{rca theorem} $(\mathsf{EA}^+)$ For any ordinal notation $\alpha$,
$$\mathbf{R}^\alpha_{\Pi^1_1(\Pi^0_3)}(\mathsf{RCA}_0) \equiv_{\boldsymbol\Pi^0_\infty} \mathbf{R}^{1+\alpha}_{\mathbf{\Pi}^0_3}(\mathsf{EA}^+(X)).$$
\end{theorem}

\begin{proof}
We prove the claim by reflexive induction. We reason in $\mathsf{EA}^+$ and assume the reflexive induction hypothesis: provably in $\mathsf{EA}^+$, for any $\beta\prec\alpha$, $$\mathbf{R}^\beta_{\Pi^1_1(\Pi^0_3)}(\mathsf{RCA}_0) \equiv_{\boldsymbol\Pi^0_\infty} \mathbf{R}^{1+\beta}_{\mathbf{\Pi}^0_3}(\mathsf{EA}^+(X)).$$  Of course, since $\mathsf{RCA}_0$ contains $\mathsf{EA}^+$, this also implies that, 
$$\mathsf{RCA}_0 \vdash \forall \beta\prec \alpha \Big( \mathbf{R}^\beta_{\Pi^1_1(\Pi^0_3)}(\mathsf{RCA}_0) \equiv_{\boldsymbol\Pi^0_\infty} \mathbf{R}^{1+\beta}_{\mathbf{\Pi}^0_3}\big(\mathsf{EA}^+(X)\big) \Big) $$

If $\mathsf{RCA}_0$ proves mutual $\Gamma$ conservation of two theories $T$ and $U$, then $\mathsf{RFN}_{\Gamma}(T)$ and $\mathsf{RFN}_{\Gamma}(U)$ are equivalent over $\mathsf{RCA}_0$. Thus, we immediately infer
\begin{equation}\label{mutualconservation}
\mathsf{RCA}_0 \vdash \forall \beta\prec \alpha \Big( \mathsf{RFN}_{\boldsymbol\Pi^0_3} \big(\mathbf{R}^\beta_{\Pi^1_1(\Pi^0_3)}(\mathsf{RCA}_0) \big) \leftrightarrow  \mathsf{RFN}_{\boldsymbol\Pi^0_3} \big(\mathbf{R}^{1+\beta}_{\mathbf{\Pi}^0_3}\big(\mathsf{EA}^+(X)\big)\big) \Big)
\end{equation}

We now reason as follows.
\begin{flalign*}
\mathbf{R}^\alpha_{\Pi^1_1(\Pi^0_3)}(\mathsf{RCA}_0) &\equiv \mathsf{RCA}_0 + \bigcup_{\beta\prec\alpha} \mathsf{RFN}_{\Pi^1_1(\Pi^0_3)}\big(\mathbf{R}^\beta_{\Pi^1_1(\Pi^0_3)}(\mathsf{RCA}_0)\big) \textrm{ by definition.}\\
&\equiv_{\boldsymbol\Pi^0_\infty} \mathsf{RCA}_0 + \bigcup_{\beta\prec\alpha} \mathsf{RFN}_{\boldsymbol\Pi^0_3}\big(\mathbf{R}^\beta_{\Pi^1_1(\Pi^0_3)}(\mathsf{RCA}_0)\big)\\
& \equiv \mathsf{RCA}_0 + \bigcup_{\beta\prec\alpha} \mathsf{RFN}_{\mathbf{\Pi}^0_3}\Big(\mathbf{R}^{1+\beta}_{\mathbf{\Pi}^0_3}\big(\mathsf{EA}^+(X)\big)\Big) \textrm{ by (\ref{mutualconservation}).} 
\end{flalign*}



Since $\mathbf{R}^1_{\mathbf{\Pi}^0_3}(\mathsf{EA}^+(X))\equiv\mathsf{I\Sigma}_1(X)$, we are able to show that $$\mathbf{R}^{1+\alpha}_{\mathbf{\Pi^0_3}}\big(\mathsf{EA}^+(X)\big)\equiv \mathsf{I\Sigma}_1(X)+\bigcup_{\beta\prec\alpha} \mathsf{RFN}_{\mathbf{\Pi}^0_3}\Big(\mathbf{R}^{1+\beta}_{\mathbf{\Pi^0_3}}\big(\mathsf{EA}^+(X)\big)\Big),$$ by the following argument:
\begin{flalign*}
\mathbf{R}^{1+\alpha}_{\mathbf{\Pi}^0_3}\big(\mathsf{EA}^+(X)\big) &\equiv \mathbf{R}^1_{\mathbf{\Pi}^0_3}\big(\mathsf{EA}^+(X)\big) + \mathbf{R}^{1+\alpha}_{\mathbf{\Pi}^0_3}\big(\mathsf{EA}^+(X)\big) \textrm{ since $1\leq 1+\alpha$.}\\
&\equiv  \mathsf{I\Sigma}_1(X)  + \mathbf{R}^{1+\alpha}_{\mathbf{\Pi}^0_3}\big(\mathsf{EA}^+(X)\big) \textrm{ since $\mathbf{R}^1_{\mathbf{\Pi}^0_3}\big(\mathsf{EA}^+(X)\big)\equiv\mathsf{I\Sigma}_1(X)$.}\\
&\equiv \mathsf{I\Sigma}_1(X)+\bigcup_{\beta\prec\alpha} \mathsf{RFN}_{\mathbf{\Pi}^0_3}\Big(\mathbf{R}^{1+\beta}_{\mathbf{\Pi^0_3}}\big(\mathsf{EA}^+(X)\big)\Big) \textrm{ by definition.}
\end{flalign*}

Hence in order to finish the proof of the lemma it will be enough to show that
$$\mathsf{I\Sigma}_1(X)+\bigcup_{\beta\prec\alpha} \mathsf{RFN}_{\mathbf{\Pi}^0_3}(\mathbf{R}^{1+\beta}_{\mathbf{\Pi^0_3}}(\mathsf{EA}^+(X)))\equiv_{\mathbf{\Pi}^0_\infty} \mathsf{RCA}_0 + \bigcup_{\beta\prec\alpha} \mathsf{RFN}_{\mathbf{\Pi}^0_3}(\mathbf{R}^{1+\beta}_{\mathbf{\Pi}^0_3}(\mathsf{EA}^+(X))),$$

which, by Lemma \ref{closure 2}, can be achieved by proving that $$\mathsf{I\Sigma}_1(X)+\bigcup_{\beta\prec\alpha} \mathsf{RFN}_{\mathbf{\Pi}^0_3}(\mathbf{R}^{1+\beta}_{\mathbf{\Pi^0_3}}(\mathsf{EA}^+(X)))$$ is closed under the $\mathbf{\Delta}^0_1$ substitution rule. We will prove this closedness in the rest of the proof. 

By a usual compactness argument, it will be enough to show that $\mathsf{I\Sigma}_1(X)$ is closed under the $\mathbf{\Delta}^0_1$ substitution rule and that for each $\beta\prec\alpha$ the theories $\mathsf{I\Sigma}_1(X)+ \mathsf{RFN}_{\mathbf{\Pi}^0_3}(\mathbf{R}^{1+\beta}_{\mathbf{\Pi^0_3}}(\mathsf{EA}^+(X)))$ are closed under the $\mathbf{\Delta}^0_1$ substitution rule. Closure of $\mathsf{I\Sigma}_1(X)$ under the $\mathbf{\Delta}^0_1$ substitution rule follows directly from Lemma \ref{closure 2}.

By Lemma \ref{closure 2}, we infer that, for each $\beta\prec\alpha$, $\mathbf{R}^{1+\beta}_{\mathbf{\Pi^0_3}}(\mathsf{EA}^+(X))$ is $\mathsf{EA}^+$ provably closed under the $\mathbf{\Delta}^0_1$ substitution rule. Thus, by Lemma \ref{closure 1}, we infer that  for each $\beta\prec\alpha$, $$\mathsf{EA}^+(X) + \mathsf{RFN}_{\mathbf{\Pi}^0_3}(\mathbf{R}^{1+\beta}_{\mathbf{\Pi^0_3}}(\mathsf{EA}^+(X)))$$ is closed under the $\mathbf{\Delta}^0_1$ substitution rule. Since $\mathsf{EA}^+(X) + \mathsf{RFN}_{\mathbf{\Pi}^0_3}(\mathbf{R}^{1+\beta}_{\mathbf{\Pi^0_3}}(\mathsf{EA}^+(X)))\sqsupseteq \mathsf{I\Sigma}_1(X)$, the theory $\mathsf{I\Sigma}_1(X)+ \mathsf{RFN}_{\mathbf{\Pi}^0_3}(\mathbf{R}^{1+\beta}_{\mathbf{\Pi^0_3}}(\mathsf{EA}^+(X)))$ is closed under the $\mathbf{\Delta}^0_1$ substitution rule. This concludes the proof of the lemma.
\end{proof}

\subsection{Iterated reflection and arithmetical comprehension}

In this subsection we investigate the relationship between iterated $\Pi^1_1$ reflection over $\mathsf{ACA}_0$ and iterated $\Pi^1_1(\Pi^0_3)$ reflection over $\mathsf{RCA}_0$. The main theorem of this subsection is that $\mathbf{R}^\alpha_{\Pi^1_1}(\mathsf{ACA}_0)$ is $\Pi^1_1(\Pi^0_3)$ conservative over $\mathbf{R}^{\varepsilon_\alpha}_{\Pi^1_1(\Pi^0_3)}(\mathsf{RCA}_0)$. The proof of the main theorem of this subsection is similar to the proof of Theorem \ref{rca theorem}. For our first step towards this result, we establish a conservation theorem relating extensions of $\mathsf{ACA}_0$ with extensions of $\mathsf{PA}(X)$.


There is a standard semantic argument that $\mathsf{ACA}_0$ is conservative over $\mathsf{PA}$ (see, e.g., \cite[Section~IX.1]{simpson2009subsystems}). We will present a version of this argument for extensions of $\mathsf{ACA}_0$ by $\Pi^1_1$ sentences. Moreover we ensure that this conservation result is provable in  $\mathsf{ACA}_0$. Before presenting the argument, we will say a bit about how we will formalize model theory within $\mathsf{ACA}_0$ for the purposes of our argument.

  We will reason in $\mathsf{ACA}_0$ and use the formalization of model theory from \cite[Section~II.8]{simpson2009subsystems}. Recall that according to formalization from \cite[Section~II.8]{simpson2009subsystems} a model $\mathfrak{M}$ essentially is a set that encodes the domain of $\mathfrak{M}$ (which is by necessity a subset of $\mathbb{N}$) and the full satisfaction relation for $\mathfrak{M}$ (the latter essentially is the elementary diagram of the model $\mathfrak{M}$). Note that if one would require $\mathfrak{M}$ contain information only about the satisfaction of atomic formulas, rather then all formulas, the resulting notion of a model would be weaker. This is due to the fact that in $\mathsf{ACA}_0$, unlike in stronger theories, it is not always possible to recover the elementary diagram of a model from its atomic diagram.

  Due to this limitation, in $\mathsf{ACA}_0$ it is sometimes (including in our proof) useful to employ weak models \cite[Definition~II.8.9]{simpson2009subsystems}. A \emph{weak model} $\mathfrak{M}$ of a theory $T$ is a set that encodes the domain of $\mathfrak{M}$ and a partial satisfaction relation for $\mathfrak{M}$ that is defined only on Boolean combinations of subformulas of formulas used in axioms of $T$ such that all the axioms of $T$ are according to this satisfaction relation. The key fact that we use is that  $\mathsf{ACA}_0$ proves that any theory that has a weak model is consistent \cite[Theorem~II.8.10]{simpson2009subsystems}.

\begin{lemma}
\label{conservativity}
$(\mathsf{ACA}_0)$ Let $\varphi(X), \psi(X)$ be $\mathbf{\Pi}^0_\infty$. If $\mathsf{ACA}_0+\forall X\;\varphi(X) \vdash \forall X\;\psi(X)$ then $\mathsf{PA}(X) + \{ \varphi[\theta] : \theta(x) \textrm{ is } \boldsymbol\Pi^0_\infty  \} \vdash \psi(X)$, where $\theta$ could contain additional variables.
\end{lemma}

\begin{proof}
We reason in $\mathsf{ACA}_0$.  We denote by $U$ the theory $\mathsf{PA}(X) + \{ \varphi[\theta] : \theta  \textrm{ is } \boldsymbol\Pi^0_\infty\}$. Let us consider any $\psi(X)$ such that $U\nvdash\psi(X)$. To prove the lemma we need to show that $\mathsf{ACA}_0+\forall X\;\varphi(X)\nvdash \forall X \;\psi(X)$.

  There is a model $\mathfrak{M}$ of $U + \neg \psi(X)$. Note that here $X$ is just a unary predicate. We enrich $\mathfrak{M}$ by adding, as the family $\mathcal{S}$ of second-order objects, all the sets defined in  $\mathfrak{M}$ by $\boldsymbol \Pi^0_{\infty}$ formulas that may contain additional parameters from the model.

  Let us first show how we could finish the proof without ensuring that our argument could be formalized in $\mathsf{ACA}_0$ and only then indicate how to carry out the formalization. Indeed, it is easy to see that the second-order structure $(\mathfrak{M},\mathcal{S})$ satisfies $\mathsf{ACA}_0+\forall X\;\varphi(X)$: the presence of the full induction schema in $U$ guarantees that $(\mathfrak{M},\mathcal{S})$ satisfies set induction, our definition of $\mathcal{S}$ guarantees that arithmetical comprehension holds in $(\mathfrak{M},\mathcal{S})$, and the fact that we had axioms $\{ \varphi[\theta] : \theta  \textrm{ is } \boldsymbol\Pi^0_\infty\}$ in $U$ guarantees that $\forall X\;\varphi(X)$ holds in  $(\mathfrak{M},\mathcal{S})$. And since $\psi(X)$ failed in $\mathfrak{M}$, the sentence $\forall X \;\psi(X)$ fails in $(\mathfrak{M},\mathcal{S})$. Therefore, $\mathsf{ACA}_0 + \forall X\;\varphi(X)\nvdash\forall X\;\psi(X)$.


   Now let us show how to formalize the latter argument in $\mathsf{ACA}_0$. We want to show that we could extend $(\mathfrak{M},\mathcal{S})$ to a weak model of $\mathsf{ACA}_0+\forall X\;\varphi(X)$.  From the satisfaction relation for $\mathfrak{M}$ we can trivially construct the partial satisfaction relation for $(\mathfrak{M},\mathcal{S})$ that covers all $\Pi^0_{\infty}$ formulas with parameters from $(\mathfrak{M},\mathcal{S})$. And since we are working in $\mathsf{ACA}_0$, using arithmetical comprehension for every (externally) fixed $n$ we could expand the latter partial satisfaction relation to all the formulas constructed from $\Pi^0_{\infty}$ formulas by arbitrary use of propositional connectives and with introduction of at most $n$ quantifier alternations. For $n=2$ this expanded partial satisfaction relation covers all the axioms of $\mathsf{ACA}_0+\forall X\;\varphi(X)+\neg\forall X\; \psi(X)$. Now after we constructed this satisfaction relation we could proceed as in the paragraph above and show that in this partial satisfaction realtion all the axioms of $\mathsf{ACA}_0+\forall X\;\varphi(X)+\neg\forall X\; \psi(X)$ are true. Hence we have a weak model of $\mathsf{ACA}_0+\forall X\;\varphi(X)$. Therefore, $\mathsf{ACA}_0 + \forall X\;\varphi(X)\nvdash \forall X\; \psi(X)$.
\end{proof}

\begin{remark}Although we don't provide a proof here, we note that with some additional care it is possible to establish Lemma \ref{conservativity} in $\mathsf{EA}^+$ by appealing to the $\Pi_2$-conservativity of $\mathsf{WKL}_0^{\star}+\mbox{``super-exponentiation is total''}$ over $\mathsf{EA}^+$, see \cite{simpson1986factorization} for the $\Pi_2$-conservativity of $\mathsf{WKL}_0^{\star}$ over $\mathsf{EA}$. But it isn't possible to prove this result in $\mathsf{EA}$ since $\mathsf{ACA}_0$ enjoys non-elementary speed-up over $\mathsf{PA}$.\end{remark}

\begin{definition}
We say that a pseudo $\Pi^1_1$ theory $T(X)$ is \emph{closed under substitution} if whenever $T\vdash \varphi(X)$ then also $T\vdash \varphi[\theta(x)]$ for any $\boldsymbol\Pi^0_{\infty}$ formula $\theta$. 
\end{definition}

\begin{lemma}\label{pure logic}
If a theory $T$ proves every substitution variant of its own axioms, then $T$ is closed under substitution. 
\end{lemma}

\begin{proof}
Suppose that $T$ proves every substitution variant of its own axioms. Let $\theta$ be a $\boldsymbol\Pi^0_{\infty}$ formula and let $\varphi(X)$ be a theorem of $T$. Since $\varphi(X)$ is a theorem of $T$, there is some finite conjunction $A_T(X)$ of axioms of $T$ such that the sentence $$A_T(X) \rightarrow \varphi(X)$$ is a theorem of pure logic. Since pure logic is closed under substitution, the sentence $$A_T[\theta(x)] \rightarrow \varphi[\theta(x)]$$ is also a theorem of pure logic. Since $T$ proves every substitution variant of its own axioms, $T$ proves $A_T[\theta(x)]$, whence $T$ proves $\varphi[\theta(x)]$.
\end{proof}

\begin{lemma}
$\mathsf{PA}(X) + \mathbf{R}^\alpha_{\mathbf{\Pi}^0_\infty}(\mathsf{PA}(X))$ is closed under substitution.
\end{lemma}

\begin{proof}
We prove the claim by reflexive induction. We reason within $\mathsf{EA}^+$ and assume the reflexive induction hypothesis: provably in $\mathsf{EA}^+$, for all $\beta\prec\alpha$, $\mathsf{PA}(X) + \mathbf{R}^\beta_{\mathbf{\Pi}^0_\infty}(\mathsf{PA}(X))$ is closed under substitution. First we note that
$$\mathbf{R}^\alpha_{\mathbf{\Pi}^0_\infty}\big(\mathsf{PA}(X)\big) \equiv \mathsf{PA}(X) + \bigcup_{\beta\prec\alpha}\mathsf{RFN}_{\mathbf{\Pi}^0_\infty}\Big(\mathbf{R}^\beta_{\mathbf{\Pi}^0_\infty}\big(\mathsf{PA}(X)\big)\Big).$$
For $\beta\prec\alpha$ let us denote by $S_\beta$ the theory $$\mathsf{PA}(X) + \mathsf{RFN}_{\mathbf{\Pi}^0_\infty}\Big(\mathbf{R}^\beta_{\mathbf{\Pi}^0_\infty}\big(\mathsf{PA}(X)\big)\Big).$$
To prove that $\mathbf{R}^\alpha_{\mathbf{\Pi}^0_\infty}\big(\mathsf{PA}(X)\big)$ is closed under substitution it suffices to prove that, for every $\beta\prec\alpha$, $S_{\beta}$ is closed under substitution.

By Lemma \ref{pure logic}, to prove that $S_\beta$ is closed under substitution, it suffices to show that $S_\beta$ proves every substitution-variant of its own axioms.  Let us use the name $U_{\beta}$ to denote the theory $\mathbf{R}^\beta_{\mathbf{\Pi}^0_\infty}\big(\mathsf{PA}(X)\big)$. An axiom of the theory $S_\beta$ is either an axiom of $\mathsf{PA}(X)$ or is a sentence of the form $\forall\vec{y}\big(Pr_{U_{\beta}}\big(\varphi(X,\vec{y})\big)\rightarrow\varphi(X,\vec{y})\big)$. Already the theory $\mathsf{PA}(X)$ proves every substitutional instance of its own axioms. By the reflexive induction hypothesis, $U_{\beta}$ is provably closed under substitution. So $S_\beta$ proves $\forall\vec{y}\big(Pr_{U_{\beta}}\big(\varphi(X,\vec{y})\big)\rightarrow\varphi(\theta,\vec{y})\big)$ for any formula $\theta$. This is to say that $S_\beta$ proves every substitution instance of its axioms.
\end{proof}

\begin{remark}
\label{reduction remark}
It follows from the lemma that the theories $\mathsf{PA}(X) + \{ \mathbf{R}^\alpha_{\mathbf{\Pi}^0_\infty}(\mathsf{PA}(X))[\theta]:\theta\in\mathbf{\Pi}^0_\infty\}$ and $\mathsf{PA}(X) + \mathbf{R}^\alpha_{\mathbf{\Pi}^0_\infty}(\mathsf{PA}(X))$ are equivalent. We will make use of this observation in the proof of Lemma \ref{key conservation lemma}.
\end{remark}

Most of the work towards proving the main theorem of this section is contained in the proof of the following key lemma.

\begin{lemma}
\label{key conservation lemma}
$\mathbf{R}^\alpha_{\Pi^1_1}(\mathsf{ACA}_0)$ is $\boldsymbol\Pi^0_\infty$ conservative over $\mathbf{R}^\alpha_{\mathbf{\Pi}^0_\infty}(\mathsf{PA}(X)).$
\end{lemma}

\begin{proof}
We prove the claim by reflexive induction. We reason within $\mathsf{ACA}_0$ and assume the reflexive induction hypothesis: provably in $\mathsf{ACA}_0$, for all $\beta\prec\alpha$, $\mathbf{R}^\beta_{\Pi^1_1}(\mathsf{ACA}_0)$ is $\boldsymbol\Pi^0_\infty$ conservative over $\mathbf{R}^\beta_{\mathbf{\Pi}^0_\infty}(\mathsf{PA}(X)).$ This means that, provably in $\mathsf{ACA}_0$, for any $\beta\prec\alpha$, $\boldsymbol\Pi^0_\infty$ reflection over $\mathbf{R}^\beta_{\Pi^1_1}(\mathsf{ACA}_0)$ is equivalent to $\boldsymbol\Pi^0_\infty$ reflection over $\mathbf{R}^\beta_{\mathbf{\Pi}^0_\infty}(\mathsf{PA}(X))$. That is:
\begin{equation}\label{mainconservation}
\mathsf{ACA}_0 \vdash \forall \beta\prec\alpha \Big( \mathsf{RFN}_{\boldsymbol\Pi^0_\infty} \big(\mathbf{R}^\beta_{\Pi^1_1}(\mathsf{ACA}_0) \big) \leftrightarrow  \mathsf{RFN}_{\boldsymbol\Pi^0_\infty} \big(\mathbf{R}^\beta_{\mathbf{\Pi}^0_\infty}(\mathsf{PA}(X)) \big) \Big)
\end{equation}

We reason as follows.
\begin{flalign*}
  \mathbf{R}^\alpha_{\Pi^1_1}(\mathsf{ACA}_0) &\equiv \mathsf{ACA}_0 + \bigcup_{\beta\prec\alpha} \mathsf{RFN}_{\Pi^1_1}\big(\mathbf{R}^\beta_{\Pi^1_1}(\mathsf{ACA}_0)\big) \textrm{ by definition.} \\
&\equiv_{\mathbf{\Pi}^0_\infty} \mathsf{ACA}_0 + \bigcup_{\beta\prec\alpha} \mathsf{RFN}_{\mathbf{\Pi}^0_\infty}\big(\mathbf{R}^\beta_{\Pi^1_1}(\mathsf{ACA}_0)\big) \\
  & \equiv \mathsf{ACA}_0 + \bigcup_{\beta\prec\alpha} \mathsf{RFN}_{\mathbf{\Pi}^0_\infty} \Big(\mathbf{R}^\beta_{\mathbf{\Pi}^0_\infty}\big( \mathsf{PA}(X)\big) \Big) \textrm{ by (\ref{mainconservation}).} \\
  & \equiv \mathsf{ACA}_0 + \bigcup_{\beta\prec\alpha} \mathsf{RFN}_{\mathbf{\Pi}^0_\infty} \Big(\mathsf{PA}(X) + \mathbf{R}^\beta_{\mathbf{\Pi}^0_\infty}\big( \mathsf{PA}(X)\big) \Big) \textrm{ by definition.} \\
  & \equiv_{\boldsymbol\Pi^0_\infty} \mathsf{PA}(X) + \bigcup_{\beta\prec\alpha\;\theta\in\boldsymbol \Pi^0_\infty} \mathsf{RFN}_{\mathbf{\Pi}^0_\infty} \Big(\mathsf{PA}(X) + \mathbf{R}^\beta_{\mathbf{\Pi}^0_\infty} \big(\mathsf{PA}(X)\big) \Big)[\theta] \textrm{ by Lemma \ref{conservativity}.}\\
  & \equiv_{\boldsymbol\Pi^0_\infty} \mathsf{PA}(X) + \bigcup_{\beta\prec\alpha} \mathsf{RFN}_{\mathbf{\Pi}^0_\infty} \Big(\mathsf{PA}(X) + \mathbf{R}^\beta_{\mathbf{\Pi}^0_\infty} \big(\mathsf{PA}(X)\big) \Big) \textrm{ by Remark \ref{reduction remark}.}\\
  & \equiv_{\boldsymbol\Pi^0_\infty} \mathbf{R}^\alpha_{\mathbf{\Pi}^0_\infty}(\mathsf{PA}(X)) \textrm{ by definition.}
\end{flalign*}
This concludes the proof.
\end{proof}

The proof of the the main theorem of this section is now straightforward, given Theorem \ref{rca theorem} and Lemma \ref{key conservation lemma}.

\begin{theorem}
\label{schmerl_ACA_0_RCA_0}
$\mathbf{R}^\alpha_{\Pi^1_1}(\mathsf{ACA}_0)$ is $\Pi^1_1(\Pi^0_3)$ conservative over $\mathbf{R}^{\varepsilon_\alpha}_{\Pi^1_1(\Pi^0_3)}(\mathsf{RCA}_0)$. 
\end{theorem}

\begin{proof}
We reason as follows.
\begin{flalign*}
\mathbf{R}^\alpha_{\Pi^1_1}(\mathsf{ACA}_0) &\equiv_{\boldsymbol \Pi^0_\infty}\mathbf{R}^\alpha_{\mathbf{\Pi}^0_\infty}(\mathsf{PA}(X)) \textrm{ by Lemma \ref{key conservation lemma}.} \\
&\equiv_{\boldsymbol \Pi^0_3}\mathbf{R}^{\varepsilon_\alpha}_{\mathbf{\Pi}^0_3}(\mathsf{EA}^+(X)) \textrm{ by Theorem \ref{generalized Schmerl}.} \\
&\equiv_{\boldsymbol \Pi^0_3}\mathbf{R}^{\varepsilon_\alpha}_{\Pi^1_1(\Pi^0_3)}(\mathsf{RCA}_0) \textrm{ by Theorem \ref{rca theorem}.}
\end{flalign*}
Note for each $\Pi^1_1(\Pi^0_3)$ sentence $\varphi$ we could find a $\boldsymbol \Pi^0_3$ sentence $\varphi'$ such that $\mathsf{RCA}_0$ proves the equivalence of $\varphi$ and (the translation into the second order language of) $\varphi'$. Thus moreover we have
$$\mathbf{R}^\alpha_{\Pi^1_1}(\mathsf{ACA}_0)\equiv_{\Pi^1_1(\Pi^0_3)} \mathbf{R}^{\varepsilon_\alpha}_{\Pi^1_1(\Pi^0_3)}(\mathsf{RCA}_0).$$
This completes the proof of the theorem.
\end{proof}


\section{Reflection ranks and proof-theoretic ordinals}
\label{proof-theoretic_ordinal}

In this section we introduce the notion of \emph{reflection rank}. We then use the results from the previous section to establish connections between reflection ranks and proof-theoretic ordinals.

\subsection{Reflection ranks}

Recall that the reflection order $\prec_{\Pi^1_1}$ on r.e. extensions of $\mathsf{ACA}_0$ is:
$$\mathit{T}_1 \prec_{\Pi^1_1} \mathit{T}_2\stackrel{\mbox{\footnotesize \textrm{def}}}{\iff} \mathit{T}_2\vdash \mathsf{RFN}_{\Pi^1_1}(\mathit{T}_1).$$
For a theory $\mathit{T}\sqsupseteq \mathsf{ACA}_0$ we define the \emph{reflection rank} $|\mathit{T}|_{\mathsf{ACA}_0}\in \mathbf{On}\cup\{\boldsymbol\infty\}$ as the rank of $\mathit{T}$ in the order $\prec_{\Pi^1_1}$.

\begin{remark} We recall that as usual the rank function $\rho\colon A \to \mathbf{On}\cup\{\boldsymbol\infty\}$ for a binary relation $(A,\triangleleft)$ is the only function such that
  $\rho(a)=\sup \{\rho(b)+1\mid b\mathrel{\triangleleft} a\}$. Here the linear order $<$ on ordinals is extended to the class $\mathbf{On}\cup\{\boldsymbol\infty\}$ by puting $\boldsymbol\alpha<\boldsymbol\infty$, for all $\boldsymbol\alpha\in \mathbf{On}$. The operation $\boldsymbol \alpha\mapsto \boldsymbol\alpha+1$ is extended to the class $\mathbf{On}\cup\{\boldsymbol\infty\}$ and puting $\boldsymbol\infty+1=\boldsymbol\infty$. Note that $\rho(a)\in \mathbf{On}$ iff the cone $\{b\mid b\mathrel{\triangleleft} a\}$ is well-founded with respect to $\triangleleft$.
\end{remark}

Recall that Theorem \ref{well-foundedness_reflection} states that $|\mathit{T}|_{\mathsf{ACA}_0}\in \mathbf{On}$, for $\Pi^1_1$-sound $\mathit{T}$.

We will also consider the more general notion of reflection rank with respect to some other base theories.  For second-order theories $\mathit{U}\sqsupseteq \mathsf{RCA}_0$ we consider the reflection order $\prec_{\Pi^1_1(\Pi^0_3)}$:
$$\mathit{U}_1 \prec_{\Pi^1_1(\Pi^0_3)} \mathit{U}_2\stackrel{\mbox{\footnotesize \textrm{def}}}{\iff} \mathit{U}_2\vdash \mathsf{RFN}_{\Pi^1_1(\Pi^0_3)}(\mathit{U}_1).$$
Let us consider some base theory $\mathit{T}_0\sqsupseteq \mathsf{RCA}_0$. We define the set $\mathcal{E}\mbox{-}\mathit{T}_0$ of all theories $\mathit{U}$ such that $\mathsf{EA}$ proves that $\mathit{U}\sqsupseteq \mathit{T}_0$. For $\mathit{U}\in \mathcal{E}\mbox{-}\mathit{T}_0$ we denote  by $|\mathit{U}|_{\mathit{T}_0}$ the rank of $\mathit{U}$ in the order $(\mathcal{E}\mbox{-}\mathit{T}_0,\prec_{\Pi^1_1(\Pi^0_3)})$. Note that $\Pi^1_1(\Pi^0_3)$-sound extensions of $T_0$ have a well-founded rank in this ordering by Theorem \ref{RCA_0_reflection}.

\begin{remark} For a theory $\mathit{T}_0$ given by a finite list of axioms the set  $\mathcal{E}\mbox{-}\mathit{T}_0$ coincides with the set of all $U$ such that $\mathit{U}\sqsupseteq \mathit{T}_0$. Indeed, for any $T_0$ given by a finite list of axioms we have a $\Sigma_1$ formula in $\mathsf{EA}$ that expresses $\mathit{U}\sqsupseteq \mathit{T}_0$  with $\mathit{U}$ as a parameter (the $\Sigma_1$ formula states that there is a $U$-proof of the conjunction of all the axioms of $T_0$). 
\end{remark}

\begin{remark} The definition of the rank $|\mathit{T}|_{\mathsf{ACA}_0}$ given in the beginning of the section coincides with the more general definition of rank, since in $\mathsf{ACA}_0$ each $\Pi^1_1$ formula is equivalent to a $\Pi^1_1(\Pi^0_3)$-formula and hence for any $\mathit{T}\sqsupseteq \mathsf{ACA}_0$,
$$\mathsf{ACA}_0\vdash \mathsf{RFN}_{\Pi^1_1}(\mathit{T})\mathrel{\leftrightarrow}\mathsf{RFN}_{\Pi^1_1(\Pi^0_3)}(\mathit{T}).$$
\end{remark}

Straightforwardly from Theorem \ref{RCA_0_reflection} we get the following.
\begin{corollary} \label{Pi^1_1-sound_rank}If $\mathit{U}\sqsupseteq \mathsf{RCA}_0$ is  $\Pi^1_1(\Pi^0_3)$-sound, then  the rank $|\mathit{U}|_{\mathsf{RCA}_0}\in \mathbf{On}$. Hence for each $T_0\sqsupseteq \mathsf{RCA}_0$ and  $\Pi^1_1(\Pi^0_3)$-sound theory $U\in \mathcal{E}\mbox{-}T_0$ we have $|U|_{T_0}\in \mathbf{On}$.
\end{corollary}

\begin{remark} \label{rank_not_Con} The converse of Corollary \ref{Pi^1_1-sound_rank} is not true, there are $\Pi^1_1(\Pi^0_3)$ unsound theories whose rank is an ordinal. In particular, for each consistent theory $T_0\sqsupseteq \mathsf{RCA}_0$, we have $|T_0+\lnot \mathsf{Con}(T_0)|_{T_0}=0$. Indeed, assume $T_0+\lnot \mathsf{Con}(T_0)\vdash \mathsf{RFN}_{\Pi^1_1}(U)$, for some $U\in \mathcal{E}\mbox{-}T_0$. Then
  $$\begin{aligned}T_0+\lnot \mathsf{Con}(T_0) & \vdash \mathsf{RFN}_{\Pi^1_1}(T_0)\\
                                              &  \vdash \mathsf{Con}(T_0)\\
                                              &  \vdash \bot.                                              
  \end{aligned}$$
  But by G\"odel's Second Incompleteness Theorem  $T_0+\lnot \mathsf{Con}(T_0) $ is consistent. This is to say that, though $T_0+\neg\mathsf{Con}(T_0)$ is not $\Pi^1_1$ sound, $|T_0+\neg\mathsf{Con}(T_0)|_{T_0}\in\mathbf{On}$.

  Note that later we will introduce a notion of robust reflection rank that enjoys much better behavior and, in particular, satisfies the converse of Corollary \ref{Pi^1_1-sound_rank}.
\end{remark}


Recall that for an ordinal notation $\alpha$ we denote by $|\alpha|\in \mathbf{On}\cup\{\boldsymbol \infty\}$ the rank of the ordinal notation $\alpha$ in the order $\prec$.

The main proposition proved in this subsection is the following:

\begin{proposition} \label{rank_of_iteration} For each $\Pi^1_2(\Pi^0_2)$-sound theory $\mathit{T}_0$ and ordinal notation $\alpha$:
  $$|\mathbf{R}_{\Pi^1_1(\Pi^0_3)}^{\alpha}(\mathit{T}_0)|_{\mathit{T}_0}=|\alpha|.$$
\end{proposition}

In order to prove the proposition we first prove the following lemma.

\begin{lemma} \label{reflection_of_iteration} If $|U|_{T_0}>|\alpha|$ then  there is a true $\Sigma^1_1(\Pi^0_2)$ sentence $\varphi$ such that \begin{equation}\label{rank_of_iteration_eq1}U+\varphi\vdash \mathsf{RFN}_{\Pi^1_1(\Pi^0_3)}(\mathbf{R}^{\alpha}_{\Pi^1_1(\Pi^0_3)}(T_0)).\end{equation}
\end{lemma}
\begin{proof} We prove the lemma by transfinite induction on $|\alpha|$. Since $|U|_{T_0}>|\alpha|$, there is a $V\in \mathcal{E}\mbox{-}T_0$ such that $U\vdash \mathsf{RFN}_{\Pi^1_1(\Pi^0_3)}(V)$ and $|V|_{T_0}\ge |\alpha|$. By the induction hypothesis there are true $\Sigma^1_1(\Pi^0_2)$ sentences $\varphi_{\beta}$, for all $\beta\prec \alpha$, such that
  $$V+\varphi_{\beta}\vdash \mathsf{RFN}_{\Pi^1_1(\Pi^0_3)}(\mathbf{R}^{\beta}_{\Pi^1_1(\Pi^0_3)}(T_0)).$$
  
  We now formalize the latter fact by a single $\Sigma^1_1(\Pi^0_2)$ sentence $\varphi$, which states that there is a sequence of $\Pi^0_2$ formulas $\langle \psi_{\beta}(Y)\mid \beta\prec \alpha\rangle$ without free variables other that $Y$ and sequence of sets $\langle S_{\beta}\mid \beta\prec \alpha\rangle$ such that
  \begin{itemize}
  \item for all $\beta$, the formula $\psi_{\beta}(Y)$ holds on $Y=S_{\beta}$;
  \item for all $\beta$, we have $V +\exists Y\;\psi_{\beta}(Y)\vdash \mathsf{RFN}_{\Pi^1_1(\Pi^0_3)}(\mathbf{R}^{\beta}_{\Pi^1_1(\Pi^0_3)}(T_0))$.
  \end{itemize}
  It is easy to see that indeed we could form a $\Sigma^1_1(\Pi^0_2)$ sentence $\varphi$ constituting the desired formalization.

  Now let us show that $\varphi$ is true. Without loss of generality, we may assume that each $\varphi_{\beta}$ is of the form $\exists Y \theta_{\beta}(Y)$, where all $\theta_{\beta}(Y)$ are $\Pi^0_2$-formulas. We put each $\psi_{\beta}$ to be $\theta_{\beta}$ and for each $\beta\prec \alpha$ we choose $S_{\beta}$ so that $\theta_{\beta}(Y)$ holds on $Y=S_{\beta}$. Thus we see that $\varphi$ is true.


We establish  (\ref{rank_of_iteration_eq1}) by reasoning in $U+\varphi$ and showing that the theory $\mathbf{R}^{\alpha}_{\Pi^1_1(\Pi^0_3)}(T_0)$ is $\Pi^1_1(\Pi^0_3)$-sound. It is enough for us to establish the $\Pi^1_1(\Pi^0_3)$-soundness of each finite subtheory of $\mathbf{R}^{\alpha}_{\Pi^1_1(\Pi^0_3)}(T_0)$, i.e., each theory
$$T_0+ \mathsf{RFN}_{\Pi^1_1(\Pi^0_3)}(\mathbf{R}^{\beta}_{\Pi^1_1(\Pi^0_3)}(T_0)),$$
for $\beta \prec \alpha$. We know (from $U$) that $V$ is $\Pi^1_1(\Pi^0_3)$-sound. And also (from $\varphi$) we have a $\Pi^0_2$-formula $\psi_{\beta}(Y)$ such that
$$V+\exists Y\; \psi_{\beta}(Y)\vdash \mathsf{RFN}_{\Pi^1_1(\Pi^0_3)}(\mathbf{R}^{\beta}_{\Pi^1_1(\Pi^0_3)}(T_0))$$
and a set $S_{\beta}$ such that $\psi_{\beta}(S_{\beta})$ holds. From the $\Pi^1_1(\Pi^0_3)$-soundness of $V$ we infer the $\Pi^1_1(\Pi^0_3)$-soundness of $V+\exists Y\; \psi_{\beta}(Y)$. Therefore $T_0+ \mathsf{RFN}_{\Pi^1_1(\Pi^0_3)}(\mathbf{R}^{\beta}_{\Pi^1_1(\Pi^0_3)}(T_0))$ is $\Pi^1_1(\Pi^0_3)$-sound. 
\end{proof}

We are nearly in a position to prove Proposition \ref{rank_of_iteration}. Before doing so, we pause to state two lemmas, the truth of which may easily be verified.

\begin{lemma}[$\mathsf{RCA}_0$]
\label{first principle}
\item If $T$ is $\Pi^1_2(\Pi^0_2)$ sound and $\alpha$ is a well-ordering, then $\mathbf{R}^\alpha_{\Pi^1_1(\Pi^0_3)}(T)$ is $\Pi^1_1(\Pi^0_3)$ sound.
\end{lemma}

\begin{lemma}[$\mathsf{RCA}_0$]
\label{second principle}
If $T$ is $\Pi^1_1(\Pi^0_3)$ sound and $\varphi$ is a true $\Sigma^1_1(\Pi^0_2)$ formula, then $T+\varphi$ is $\Pi^1_1(\Pi^0_3)$ sound.
\end{lemma}

\begin{proof} First let us notice that $|\mathbf{R}_{\Pi^1_1(\Pi^0_3)}^{\alpha}(\mathit{T}_0)|_{\mathit{T}_0}\ge|\alpha|$. Indeed this inequality holds since there is a homomorphism $\beta\mapsto \mathbf{R}_{\Pi^1_1(\Pi^0_3)}^{\beta}(\mathit{T}_0)$ of the low $\prec$-cone of $\alpha$ (the order $(\{\beta\mid \beta\preceq\alpha\},\prec)$) to the low $\prec_{\Pi^1_1(\Pi^0_3)}$-cone of $\mathbf{R}_{\Pi^1_1(\Pi^0_3)}^{\alpha}(\mathit{T}_0)$ in $\mathcal{E}\mbox{-}\mathit{T}_0$.

  Now assume for a contradiction that $|\mathbf{R}_{\Pi^1_1(\Pi^0_3)}^{\alpha}(\mathit{T}_0)|_{\mathit{T}_0}>|\alpha|$. In this case by Lemma \ref{reflection_of_iteration} we have $$\mathbf{R}_{\Pi^1_1(\Pi^0_3)}^{\alpha}(\mathit{T}_0)+\varphi\vdash \mathsf{RFN}_{\Pi^1_1(\Pi^0_3)}(\mathbf{R}_{\Pi^1_1(\Pi^0_3)}^{\alpha}(\mathit{T}_0)),$$ for some true $\Sigma^1_1(\Pi^0_2)$ sentence $\varphi$. We derive
  $$\begin{aligned} \mathbf{R}_{\Pi^1_1(\Pi^0_3)}^{\alpha}(\mathit{T}_0)+\varphi&\vdash \mathsf{RFN}_{\Pi^1_1(\Pi^0_3)}(\mathbf{R}_{\Pi^1_1(\Pi^0_3)}^{\alpha}(\mathit{T}_0)+\varphi)\\
    & \vdash \mathsf{Con}(\mathbf{R}_{\Pi^1_1(\Pi^0_3)}^{\alpha}(\mathit{T}_0)+\varphi).\end{aligned}$$
 So $\mathbf{R}_{\Pi^1_1(\Pi^0_3)}^{\alpha}(\mathit{T}_0)+\varphi$ is inconsistent by G\"odel's Second Incompleteness Theorem.
 Yet by Lemma \ref{first principle}, since $T_0$ is $\Pi^1_2(\Pi^0_2)$ sound, $\mathbf{R}_{\Pi^1_1(\Pi^0_3)}^{\alpha}(\mathit{T}_0)$ is $\Pi^1_1(\Pi^0_3)$ sound. Thus, by Lemma \ref{second principle}, $\mathbf{R}_{\Pi^1_1(\Pi^0_3)}^{\alpha}(\mathit{T}_0)+\varphi$ is consistent. This is a contradiction.
\end{proof}

\subsection{Proof-theoretic ordinals}

For a theory $T\sqsupseteq \mathsf{RCA}_0$ we write $|T|_{\mathsf{WO}}$ to denote the \emph{ proof-theoretic ordinal} of $T$, which we define as the supremum of the ranks $|\alpha|$ of ordinal notations $\alpha$ such that $T\vdash \mathsf{WO}(\alpha)$. The formula $\mathsf{WO}(\alpha)$ is $$\forall X((\exists \beta\prec \alpha)\;\beta\in X\to (\exists \beta\prec\alpha)(\beta\in X\land (\forall \gamma\prec \beta)\gamma\not\in X)).$$

\begin{remark} One may also define $|T|_{\mathsf{WO}}$ for second-order theories  in terms of primitive recursive well-orders (alternatively recursive well-orders), i.e., $|T|_{\mathsf{WO}}$ then would be defined as  the supremum of order types of primitive recursive ($T$-provably recursive) binary relations $\triangleright$ such that $T\vdash \mathsf{WO}(\triangleright)$. If $T$ proves the well-orderedness of an ill-founded relation then this supremum by definition is $\boldsymbol \infty$. We note that our definition coincides with the definitions above for $T\sqsupseteq \mathsf{RCA}_0$. The connection between presentations of ordinals of various degrees of ``niceness'' is extensively discussed in M.~Rathjen's survey \cite{rathjen1999realm}, and the equivalence under consideration could be proved by a slight extension of the proof of \cite[Proposition~2.19(i)]{rathjen1999realm}.\footnote{The proof of \cite[Proposition~2.19(i)]{rathjen1999realm} implicitly uses $\Sigma_1$-collection inside the theory $T$, although the claim is stated for all $T$ containing $\mathsf{PRA}$. But this issue doesn't affect the theories that we are interested in since $\mathsf{RCA}_0\vdash \mathsf{B\Sigma}_1$}
  \end{remark}

\begin{theorem} \label{ACA_0_ordinal_analysis}$|\mathbf{R}^{\alpha}_{\Pi^1_1}(\mathsf{ACA}_0)|_{\mathsf{WO}}=|\varepsilon_{\alpha}|$.
\end{theorem}

In order to prove the theorem we first establish the following lemma:
\begin{lemma} \label{wo_to_iterated_special_case} For each $\alpha$
  \begin{enumerate}
  \item the theory $\mathsf{ACA}_0$ proves $\mathsf{WO}(\alpha)\to \mathsf{RFN}_{\Pi^1_1(\Pi^0_3)}(\mathbf{R}^{\alpha}_{\Pi^1_1(\Pi^0_3)}(\mathsf{RCA}_0))$;
  \item the theory $\mathsf{ACA}_0^{+}$ proves $\mathsf{WO}(\alpha)\to \mathsf{RFN}_{\Pi^1_1}(\mathbf{R}^{\alpha}_{\Pi^1_1}(\mathsf{ACA}_0))$.
  \end{enumerate}
  
\end{lemma}

We will derive Lemma \ref{wo_to_iterated_special_case} from the more general Lemma \ref{wo_to_iterated_general}.

We will follow Simpson's formalization of countable coded models of the language of second-order arithmetic \cite[Definition~VII.2.1]{simpson2009subsystems}. Under this definition a countable coded $\omega$-model $\mathfrak{M}$ is a code for a countable family $W_0,W_1,\ldots$ of subsetes of $\mathbb{N}$, where $\{W_0,W_1,\ldots\}$ is the $\mathfrak{M}$-domain for sets of naturals.  We note that the property ``$\mathfrak{M}$ is a countable coded $\omega$-model'' is arithmetical. The expression $X\in \mathfrak{M}$ denotes the natural $\Sigma^0_2$ formula that expresses the fact that the set $X$ is coded in a model $\mathfrak{M}$ (i.e. it is one of $X=W_i$, for some $i$). For each fixed second-order formula $\varphi(X_1,\ldots,X_n,x_1,\ldots,x_n)$ the expression $\mathfrak{M}\models \varphi(X_1,\ldots,X_n,x_1,\ldots,x_n)$ denotes the natural second-order formula that expresses that $\mathfrak{M}$ is a countable coded $\omega$-model, sets $X_1,\ldots,X_n$ are coded in $\mathfrak{M}$, and $\varphi(X_1,\ldots,X_n,x_1,\ldots,x_n)$ is true in $\mathfrak{M}$. We express the fact that that $\varphi(X_1,\ldots,X_n,x_1,\ldots,x_n)$ is true in $\mathfrak{M}$ by relativizing second-order quatifiers $\forall X$ and $\exists X$ to $\forall X\in \mathfrak{M}$ and $\exists X\in\mathfrak{M}$. Note that the latter quantifiers are in fact just first-order quantifiers. Hence $\mathfrak{M}\models \varphi(\vec{X},\vec{x})$ is equivalent to a $\Pi^0_m$-formula, where $m$ depends only on the depth of quantifier alternations in $\varphi$. For a fixed theory $T$ given by a finite list of axioms, by $\mathfrak{M}\models T$ we mean the formula $\mathfrak{M}\models \varphi$, where $\varphi$ is the conjunction of all the axioms of $T$.

For each theory $T_0\sqsupseteq \mathsf{RCA}_0$ given by a finite list of axioms we denote by $T_0^{+}$ the theory $T_0+\mathsf{ACA}_0+\mbox{``every set is contained in an $\omega$-model of $T_0$.''}$ We use this notation by analogy with $\mathsf{ACA}_0^+$. We note that for $T_0=\mathsf{ACA}_0$ the theory $T_0^{+}$ is just $\mathsf{ACA}_0^+$ and for $T_0=\mathsf{RCA}_0$ the theory $T_0^{+}$ is just $\mathsf{ACA}_0$.

\begin{lemma}\label{wo_to_iterated_general}  For each $T_0$ given by a finite list of axioms $$T_0^+\vdash\forall \alpha \; \Big( \mathsf{WO}(\alpha)\to \mathsf{RFN}_{\Pi^1_1(\Pi^0_3)}\big(\mathbf{R}^{\alpha}_{\Pi^1_1(\Pi^0_3)}(T_0)\big)\Big).$$
\end{lemma}
\begin{proof}  We reason in $T_0^+$. We assume $\mathsf{WO}(\alpha)$ and claim $\mathsf{RFN}_{\Pi^1_1(\Pi^0_3)}(\mathbf{R}^{\alpha}_{\Pi^1_1(\Pi^0_3)}(T_0))$.

  Note that it suffices to show that $\mathsf{RFN}_{\Pi^1_1(\Pi^0_3)}(\mathbf{R}^{\alpha}_{\Pi^1_1(\Pi^0_3)}(T_0))$ is true in all the $\omega$-models of $T_0$. Indeed, since $\mathsf{RFN}_{\Pi^1_1(\Pi^0_3)}(\mathbf{R}^{\alpha}_{\Pi^1_1(\Pi^0_3)}(T_0))$ is a $\Pi^1_1(\Pi^0_3)$ sentence, if it fails, this fact is witnessed by some set $X$ and hence $\mathsf{RFN}_{\Pi^1_1(\Pi^0_3)}(\mathbf{R}^{\alpha}_{\Pi^1_1(\Pi^0_3)}(T_0))$ fails in all the $\omega$-models of $T_0$ containing $X$.

  Now let us consider an $\omega$-model $\mathfrak{M}$ of $T_0$ and show $\mathsf{RFN}_{\Pi^1_1(\Pi^0_3)}(\mathbf{R}^{\alpha}_{\Pi^1_1(\Pi^0_3)}(T_0))$.  We note that, for some fixed $k$, all the facts of the form $\mathfrak{M}\models \mathsf{RFN}_{\Pi^1_1(\Pi^0_3)}(\mathbf{R}^{\beta}_{\Pi^1_1(\Pi^0_3)}(T_0))$ are $\Pi^0_k$. In order to finish the proof it suffices to show $\mathfrak{M}\models\mathsf{RFN}_{\Pi^1_1(\Pi^0_3)}(\mathbf{R}^{\beta}_{\Pi^1_1(\Pi^0_3)}(T_0))$, for all $\beta\preceq \alpha$ by transfinite induction on $\beta\preceq \alpha$. By the induction hypothesis we know that $\mathfrak{M}$ is a model of $\mathbf{R}^{\beta}_{\Pi^1_1(\Pi^0_3)}(T_0)$. Since $\mathfrak{M}$ is an $\omega$-model we need to show that for all the (standard) proofs $p$ of a $\Pi^1_1(\Pi^0_3)$-sentence $\varphi$ in $\mathbf{R}^{\beta}_{\Pi^1_1(\Pi^0_3)}(T_0)$ the sentence $\varphi$ is true in $\mathfrak{M}$. We consider some proof $p$ of this form and apply the cut-elimination theorem for predicate calculus to make sure that all the intermediate formulas in the proof are of the complexity  $\Pi^1_n(\Pi^0_m)$ for some externally fixed $n$ and $m$ (depending only on the complexity of the axioms of $T_0$). We proceed by showing by induction on formulas in the proof that all of them are true in the model $\mathfrak{M}$; we can do this since the satisfaction relation for $\Pi^1_n(\Pi^0_m)$-formulas in $\mathfrak{M}$ is arithmetical.\end{proof}

\begin{lemma} \label{it_ref_to_wo} $$\mathsf{RCA}_0\vdash \forall \alpha\;\Big(\mathsf{RFN}_{\Pi^1_1(\Pi^0_3)}\big(\mathbf{R}_{\Pi^1_1(\Pi^0_3)}^{\alpha}(\mathsf{RCA}_0)\big)\to \mathsf{WO}(\alpha)\Big)$$
\end{lemma}
\begin{proof} We prove the lemma by reflexive induction on $\alpha$ in $\mathsf{RCA}_0$. We reason in $\mathsf{RCA}_0$ and assume the reflexive induction hypothesis $$\forall \beta\prec \alpha\; \mathsf{Pr}_{\mathsf{RCA}_0}\Big(\mathsf{RFN}_{\Pi^1_1(\Pi^0_3)}\big(\mathbf{R}_{\Pi^1_1(\Pi^0_3)}^{\beta}(\mathsf{RCA}_0)\big)\to \mathsf{WO}(\beta)\Big).$$ We need to show that:
\begin{equation} \label{conditional}
\mathsf{RFN}_{\Pi^1_1(\Pi^0_3)}\big(\mathbf{R}_{\Pi^1_1(\Pi^0_3)}^{\alpha}(\mathsf{RCA}_0)\big)\to \mathsf{WO}(\alpha)
\end{equation}
So assume the antecedent of (\ref{conditional}). From the reflexive induction hypothesis we see that for each individual  $\beta\prec \alpha$ the theory $\mathbf{R}_{\Pi^1_1(\Pi^0_3)}^{\alpha}(\mathsf{RCA}_0)$ proves $\mathsf{WO}(\beta)$. Since $\mathsf{WO}(\beta)$ is a $\Pi^1_1(\Pi^0_3)$-formula, we infer from the antecedent of (\ref{conditional}) that $\forall \beta\prec \alpha \; \mathsf{WO}(\beta)$. Thus $\mathsf{WO}(\alpha)$.
\end{proof}

Now we are ready to prove Theorem \ref{ACA_0_ordinal_analysis}
\begin{proof} From Theorem \ref{schmerl_ACA_0_RCA_0} we know that
  $$\mathbf{R}^{\alpha}_{\Pi^1_1}(\mathsf{ACA}_0)\equiv_{\Pi^1_1(\Pi^0_3)} \mathbf{R}^{\varepsilon_{\alpha}}_{\Pi^1_1(\Pi^0_3)}(\mathsf{RCA}_0).$$

  From Lemma \ref{it_ref_to_wo} we see that $\mathbf{R}^{\varepsilon_{\alpha}}_{\Pi^1_1(\Pi^0_3)}(\mathsf{RCA}_0)$ proves $\mathsf{WO}(\beta)$ for each $\beta\prec \varepsilon_\alpha$ and thus $|\mathbf{R}^{\alpha}_{\Pi^1_1}(\mathsf{ACA}_0)|_{\mathsf{WO}}\ge|\varepsilon_{\alpha}|$.

  In order to prove $|\mathbf{R}^{\alpha}_{\Pi^1_1}(\mathsf{ACA}_0)|\le|\varepsilon_{\alpha}|$ let us assume that for some $\beta$ the theory $\mathbf{R}^{\alpha}_{\Pi^1_1}(\mathsf{ACA}_0)$ proves $\mathsf{WO}(\beta)$ and then show that $|\beta|< |\varepsilon_{\alpha}|$. Indeed, by Lemma \ref{wo_to_iterated_special_case} the theory $\mathbf{R}^{\alpha}_{\Pi^1_1}(\mathsf{ACA}_0)$  proves $\mathsf{RFN}_{\Pi^1_1(\Pi^0_3)}(\mathbf{R}^{\beta}_{\Pi^1_1(\Pi^0_3)}(\mathsf{RCA}_0))$.  Hence  $$|\beta|<|\mathbf{R}^{\alpha}_{\Pi^1_1}(\mathsf{ACA}_0)|_{\mathsf{RCA}_0}=|\mathbf{R}^{\varepsilon_{\alpha}}_{\Pi^1_1(\Pi^0_3)}(\mathsf{RCA}_0)|_{\mathsf{RCA}_0}.$$ And  Proposition \ref{rank_of_iteration} gives us
  $$|\varepsilon_{\alpha}|=|\mathbf{R}^{\varepsilon_{\alpha}}_{\Pi^1_1(\Pi^0_3)}(\mathsf{RCA}_0)|_{\mathsf{RCA}_0}>|\beta|.$$
  This completes the proof.
\end{proof}

\subsection{Extensions of $\mathsf{ACA}_0^+$}

It is usually attributed to Kreisel that for extensions $T\sqsupseteq\mathsf{ACA}_0$ the proof-theoretic ordinal $|T|_{\mathsf{WO}}=|T+\varphi|_{\mathsf{WO}}$, for any true $\Sigma^1_1$-sentence $\varphi$ (see \cite[Theorem~6.7.4,6.7.5]{pohlers2008proof}). We note that our notion of reflection rank $|T|_{\mathsf{ACA}_0}$ does not enjoy the same property.
\begin{remark} Let us consider an ordinal notation system $\alpha$ for some large recursive ordinal, for example the Bachmann-Howard ordinal. 
  Now we modify $\alpha$ to define pathological ordinal notation $\alpha'$. The order $\prec_{\alpha'}$ is the restriction of $\prec_{\alpha}$ to numbers $m$ such that $\forall x\le m\;\lnot\mathsf{Prf}_{\mathsf{ACA}_0}(x,0=1)$. And $\alpha'$ corresponds to the same element of the domain of $\prec_{\alpha}$ as $\alpha$ (note that since $\mathsf{ACA}_0$ is consistent this element is in the domain of $\prec_{\alpha'}$ as well). We see externally that $\alpha'$ is isomorphic to $\alpha$, since $\mathsf{ACA}_0$ is consistent. Let us denote by $\mathsf{Iso}$ the true $\Sigma^1_1$-sentence that expresses the fact that $\alpha$ and $\alpha'$ are isomorphic. Clearly, $$\mathsf{ACA}_0+\mathsf{WO}(\alpha')+\mathsf{Iso}\sqsupseteq \mathsf{ACA}_0+\mathsf{WO}(\alpha),$$ $$|\mathsf{ACA}_0+\mathsf{WO}(\alpha')+\mathsf{Iso}|_{\mathsf{ACA}_0}\ge|\mathsf{ACA}_0+\mathsf{WO}(\alpha)|_{\mathsf{ACA}_0}$$ and under our choice of $\alpha$ the rank $|\mathsf{ACA}_0+\mathsf{WO}(\alpha)|_{\mathsf{ACA}_0}$ will be equal to the Bachmann-Howard ordinal. At the same time, the theory $\mathsf{ACA}_0+\lnot\mathsf{Con}(\mathsf{ACA}_0)$ proves that $\alpha'$ is isomorphic to some finite order and hence $$\mathsf{ACA}_0+\lnot\mathsf{Con}(\mathsf{ACA}_0)\vdash \mathsf{WO}(\alpha').$$ Hence $$|\mathsf{ACA}_0+\mathsf{WO}(\alpha')|_{\mathsf{ACA}_0}\le |\mathsf{ACA}_0+\lnot \mathsf{Con}(\mathsf{ACA}_0)|_{\mathsf{ACA}_0}=0,$$ the latter equality follows from Remark \ref{rank_not_Con}. And thus $$|\mathsf{ACA}_0+\mathsf{WO}(\alpha')|_{\mathsf{WO}}<|\mathsf{ACA}_0+\mathsf{WO}(\alpha')+\mathsf{Iso}|_{\mathsf{WO}}.$$ Accordingly, $\mathsf{Iso}$ is a true $\Sigma^1_1$ sentence that alters the reflection rank of the theory $\mathsf{ACA}_0+\mathsf{WO}(\alpha')$.
  

\end{remark}

We address this problem with two different results. First in Theorem \ref{omega_model_reflection_rank} we show that for any extension $T\sqsupseteq \mathsf{ACA}_0^+$, $|T|_{\mathsf{ACA}_0}=|T|_{\mathsf{WO}}$. Second we introduce the notion of robust reflection rank $|\cdot|_{\mathsf{ACA}_0}^{\star}$ that enjoys a number of nice properties and at the same time coincides with reflection rank $|\cdot|_{\mathsf{ACA}_0}$, for many natural theories $T$ (in particular, for any any $T$ such that $T\equiv_{\Pi^1_1}\mathbf{R}_{\Pi^1_1}^{\alpha}(\mathsf{ACA}_0)$, for some ordinal notation $\alpha$).

\begin{theorem} \label{omega_model_reflection_rank}Suppose $T\sqsupseteq \mathsf{ACA}_0^{+}$ then  $$|T|_{\mathsf{WO}}=|T|_{\mathsf{ACA}_0}.$$
\end{theorem}

We prove the following general theorem
\begin{theorem} Suppose $\Pi^1_2(\Pi^0_2)$-sound theory $T_0$ is given by a finite list of axioms. Then for each $U\sqsupseteq T_0^+$ we have
  $$|U|_{\mathsf{WO}}=|U|_{T_0}.$$
\end{theorem}
\begin{proof} Combining Lemma \ref{wo_to_iterated_general} and Proposition \ref{rank_of_iteration}  we see that $|U|_{\mathsf{WO}}\le |U|_{T_0}$. In order to show that $|U|_{\mathsf{WO}}\ge |U|_{T_0}$ we prove that for each $\boldsymbol \alpha<|U|_{T_0}$ we have $\boldsymbol \alpha<|U|_{\mathsf{WO}}$. We consider $\boldsymbol \alpha<|U|_{T_0}$. From Lemma \ref{reflection_of_iteration} we see that there is an ordinal notation $\alpha$ and a true $\Sigma^1_1(\Pi^0_2)$-sentence $\varphi$ such that $|\alpha|=\boldsymbol \alpha$ and $$U+\varphi\vdash \mathsf{RFN}_{\Pi^1_1(\Pi^0_3)}(\mathbf{R}^{\alpha}_{\Pi^1_1(\Pi^0_3)}(T_0)).$$ Since $T_0\sqsupseteq \mathsf{RCA}_0$, we have
  $$U+\varphi\vdash \mathsf{RFN}_{\Pi^1_1(\Pi^0_3)}(\mathbf{R}^{\alpha}_{\Pi^1_1(\Pi^0_3)}(\mathsf{RCA}_0)).$$
  And hence by Lemma \ref{it_ref_to_wo} we have $U+\varphi\vdash \mathsf{WO}(\alpha)$. Thus $$\boldsymbol\alpha=|\alpha|< |U+\varphi|_{\mathsf{WO}}=|U|_{\mathsf{WO}}.$$
  This completes the proof of the theorem.
\end{proof}

\subsection{Robust reflection rank}

The \emph{robust reflection rank} $|U|^{\star}_{T_0}$ of a theory $U\in \mathcal{E}\mbox{-}T_0$ over a theory $T_0\sqsupseteq \mathsf{RCA}_0$ is defined as follows:
$$|U|^{\star}_{T_0}=\sup\{|U+\varphi|_{T_0}\colon \varphi\mbox{ is a true $\Sigma^1_1(\Pi^0_2)$-sentence}\}.$$

\begin{proposition} \label{robust_rank_and_soundness}For theories $T_0\sqsupseteq\mathsf{RCA}_0$ and $U\in \mathcal{E}\mbox{-}T_0$ the robust reflection rank $|U|^{\star}_{T_0}$ is an ordinal iff $U$ is $\Pi^1_1(\Pi^0_3)$-sound.
\end{proposition}
\begin{proof} If $U$ is $\Pi^1_1(\Pi^0_3)$-sound then for any true $\Sigma^1_1(\Pi^0_2)$-sentence $\varphi$ the theory $U+\varphi$ is $\Pi^1_1(\Pi^0_3)$-sound. Thus, by Corollary \ref{Pi^1_1-sound_rank} each rank $|U+\varphi|_{T_0}\in\mathbf{On}$ and so  $|U|^\star_{T_0}\in\mathbf{On}$.

  If $U$ is not $\Pi^1_1(\Pi^0_3)$-sound then there is a false $\Pi^1_1(\Pi^0_3)$ sentence $\varphi$ that $U$ proves. Let $\psi$ be a true $\Sigma^1_1(\Pi^0_2)$-sentence that is $\mathsf{RCA}_0$-provably equivalent to $\lnot \varphi$. Clearly, $U+\psi$ is inconsistent, so $U+\psi \prec_{\Pi^1_1(\Pi^0_3)}U+\psi$ and hence $\boldsymbol\infty=|U+\psi|_{T_0}=|U|^{\star}_{T_0}$.\end{proof}

\begin{proposition} Suppose $T_0\sqsupseteq\mathsf{RCA}_0$ is $\Pi^1_2(\Pi^0_3)$-sound, $U\in \mathcal{E}\mbox{-}T_0$, and for some ordinal notation $\alpha$ we have $U\equiv_{\Pi^1_1(\Pi^0_3)}\mathbf{R}_{\Pi^1_1(\Pi^0_3)}^{\alpha}(T_0)$. Then $$|U|^{\star}_{T_0}=|U|_{T_0}=|\alpha|.$$
\end{proposition}
\begin{proof}We use Proposition \ref{rank_of_iteration} and see that
  $$|U|^{\star}_{T_0}\ge|U|_{T_0}=|\alpha|.$$
  Let us assume for a contradiction that $|U|^{\star}_{T_0}>|\alpha|$. In this case from Lemma \ref{reflection_of_iteration} there is a true $\Sigma^1_1(\Pi^0_2)$ sentence $\varphi$ such that
  $$U+\varphi\vdash\mathsf{RFN}_{\Pi^1_1(\Pi^0_3)}(\mathbf{R}^{\alpha}_{\Pi^1_1(\Pi^0_3)}(U)).$$
  Of course, this implies that
    $$U+\vdash \varphi \rightarrow \mathsf{RFN}_{\Pi^1_1(\Pi^0_3)}(\mathbf{R}^{\alpha}_{\Pi^1_1(\Pi^0_3)}(U)).$$
Note that $\varphi \rightarrow \mathsf{RFN}_{\Pi^1_1(\Pi^0_3)}(\mathbf{R}^{\alpha}_{\Pi^1_1(\Pi^0_3)}(U))$ is a $\Pi^1_1(\Pi^0_3)$ sentence. Thus, from the assumption that $U\equiv_{\Pi^1_1(\Pi^0_3)}\mathbf{R}_{\Pi^1_1(\Pi^0_3)}^{\alpha}(T_0)$, it follows that:
  \begin{flalign*}
  \mathbf{R}_{\Pi^1_1(\Pi^0_3)}^{\alpha}(T_0) &\vdash \varphi \rightarrow \mathsf{RFN}_{\Pi^1_1(\Pi^0_3)}(\mathbf{R}^{\alpha}_{\Pi^1_1(\Pi^0_3)}(U))\\
 \mathbf{R}^\alpha_{\Pi^1_1(\Pi^0_3)}(T_0)+\varphi & \vdash \mathsf{RFN}_{\Pi^1_1(\Pi^0_3)}(\mathbf{R}^\alpha_{\Pi^1_1(\Pi^0_3)}(U))\\
 &  \vdash \mathsf{RFN}_{\Pi^1_1(\Pi^0_3)}(\mathbf{R}^\alpha_{\Pi^1_1(\Pi^0_3)}(U)+\varphi) \textrm{ by Lemma \ref{second principle}.}\\
 &  \vdash \mathsf{Con}(\mathbf{R}^\alpha_{\Pi^1_1(\Pi^0_3)}(U)+\varphi))
\end{flalign*}
Thus, $ \mathbf{R}^\alpha_{\Pi^1_1(\Pi^0_3)}(T_0)+\varphi $ is inconsistent by G\"{o}del's Second Incompleteness Theorem.
On the other hand, since $T_0$ is $\Pi^1_2(\Pi^0_3)$ sound, $\mathbf{R}_{\Pi^1_1(\Pi^0_3)}^{\alpha}(T_0)$ is $\Pi^1_1(\Pi^0_3)$ sound by Lemma \ref{first principle}. Thus,  $ \mathbf{R}^\alpha_{\Pi^1_1(\Pi^0_3)}(T_0)+\varphi $ is consistent by Lemma \ref{second principle}. This is a contradiction.
\end{proof}

Finally we connect the notions of robust reflection rank $|\cdot|_{\mathsf{ACA}_0}^{\star}$ and proof-theoretic ordinal $|\cdot|_{\mathsf{WO}}$:
\begin{theorem} \label{robust_rank_theorem} For any theory $T\in \mathcal{E}\mbox{-}\mathsf{ACA}_0$ with robust reflection rank $|T|^{\star}_{\mathsf{ACA}_0}=\boldsymbol\alpha$ we have $|T|_{\mathsf{WO}}=\boldsymbol\varepsilon_{\boldsymbol \alpha}$ (here by definition we put $\boldsymbol\varepsilon_{\boldsymbol \infty}=\boldsymbol \infty$).
\end{theorem}
\begin{proof} First let us show that $|T|_{\mathsf{WO}}\ge \boldsymbol\varepsilon_{\boldsymbol \alpha}$. We break into cases based on whether $\boldsymbol \alpha=\infty$ or $\boldsymbol \alpha \in \mathbf{On}$

  Assume $\boldsymbol \alpha=\infty$. Then by Proposition \ref{robust_rank_and_soundness} there is false $\Pi^1_1$ sentence $\varphi$ that is provable in $T$. Now we could construct an ordinal notation $\alpha$ such that $\mathsf{WO}(\alpha)$ is $\mathsf{ACA}_0$-provably equivalent to $\varphi$: we put $\varphi$ in the tree normal form \cite[Lemma~V.1.4]{simpson2009subsystems} and take $\alpha$ to be the Kleene-Brouwer order on the tree. Clearly, $T\vdash \mathsf{WO}(\alpha)$ and $|\alpha|=\infty$. Thus $|T|_{\mathsf{WO}}=\infty=\boldsymbol\varepsilon_{\boldsymbol \alpha}$.

Now assume that $\boldsymbol \alpha\in \mathbf{On}$. Let us consider some $\boldsymbol \beta< \boldsymbol\varepsilon_{\boldsymbol \alpha}$ and show that $|T|_{\mathsf{WO}}>\boldsymbol \beta$. From the definition of robust reflection rank it is easy to see that we could find some true $\Sigma^1_1(\Pi^0_2)$ sentence $\varphi$ such that  $\boldsymbol \beta<\boldsymbol \varepsilon_{\;|T+\varphi|_{\mathsf{ACA}_0}}$. Since $|T+\varphi|_{\mathsf{ACA}_0}$ is the rank of a $\Sigma^0_1$ binary relation,  $|T+\varphi|_{\mathsf{ACA}_0}<\boldsymbol{\omega_1^{CK}}$. Thus we could choose an ordinal notation $\gamma$ such that $|\gamma|<|T+\varphi|_{\mathsf{ACA}_0}$ but $\boldsymbol \beta<\boldsymbol \varepsilon_{|\gamma|+1}$. From  Lemma \ref{reflection_of_iteration} we infer that there is a true $\Sigma^1_1(\Pi^0_2)$-sentence $\varphi'$ such that $T+\varphi+\varphi'\vdash \mathsf{RFN}_{\Pi^1_1}(\mathbf{R}^{\gamma}_{\Pi^1_1}(\mathsf{ACA}_0))$. We find a $\beta\prec \varepsilon_{\gamma+1}$ such that $|\beta|=\boldsymbol \beta$. By the same reasoning as in the proof of Theorem \ref{ACA_0_ordinal_analysis}  we infer that $\mathbf{R}^{\gamma+1}_{\Pi^1_1}(\mathsf{ACA}_0)\vdash \mathsf{WO}(\beta)$. Thus $T+\varphi+\varphi'\vdash \mathsf{WO}(\beta)$. Hence $|T+\varphi+\varphi'|_{\mathsf{WO}}>\boldsymbol\beta$. From Kreisel's Theorem about proof-theoretic ordinals of extensions of $\mathsf{ACA}_0$ we infer that $|T|_{\mathsf{WO}}=|T+\varphi+\varphi'|_{\mathsf{WO}}>\boldsymbol\beta$. 

  Now let us show that $|T|_{\mathsf{WO}}\le \boldsymbol\varepsilon_{\boldsymbol \alpha}$. Assume, for the sake of contradiction, that $|T|_{\mathsf{WO}} > \boldsymbol\varepsilon_{\boldsymbol \alpha}$. Then there is an ordinal notation $\beta$ with  $|\beta|=\boldsymbol\varepsilon_{\boldsymbol \alpha}$ such that $|T|_{\mathsf{WO}}\vdash \mathsf{WO}(\beta)$. Let us fix some ordinal notation $\alpha$ such that $|\alpha|=\boldsymbol \alpha$. Clearly, there is an isomorphism between $\beta$ and $\varepsilon_{\alpha}$. Let us denote by $\mathsf{Iso}$ the natural $\Sigma^1_1(\Pi^0_2)$-sentence expressing the latter fact. We see that $T+\mathsf{Iso}\vdash \mathsf{WO}(\varepsilon_{\alpha})$. Thus by Lemma \ref{wo_to_iterated_special_case} we see that $$T+\mathsf{Iso}\vdash \mathsf{RFN}_{\Pi^1_1(\Pi^0_3)}(\mathbf{R}^{\varepsilon_{\alpha}}_{\Pi^1_1(\Pi^0_3)}(\mathsf{RCA}_0)).$$ From Theorem \ref{schmerl_ACA_0_RCA_0} we conclude that $$T+\mathsf{Iso}\vdash \mathsf{RFN}_{\Pi^1_1(\Pi^0_3)}(\mathbf{R}^{\alpha}_{\Pi^1_1(\Pi^0_3)}(\mathsf{ACA}_0)).$$ Since over $\mathsf{ACA}_0$ every $\Pi^1_1$-formula is equivalent to a $\Pi^1_1(\Pi^0_3)$-formula, $$T+\mathsf{Iso}\vdash \mathsf{RFN}_{\Pi^1_1}(\mathbf{R}^{\alpha}_{\Pi^1_1}(\mathsf{ACA}_0)).$$ Therefore $$|T|_{\mathsf{ACA}_0}^{\star}\ge |T+\mathsf{Iso}|_{\mathsf{ACA}_0}>|\mathbf{R}^{\alpha}_{\Pi^1_1}(\mathsf{ACA}_0)|_{\mathsf{ACA}_0}=|\alpha|=\boldsymbol\alpha,$$
but $|T|_{\mathsf{ACA}_0}^{\star}= \boldsymbol\alpha$, a contradiction.\end{proof}


\section{Ordinal notation systems based on reflection principles}
\label{ref_not_sect}

In this section we turn to ordinal notation systems based on reflection principles, like the one Beklemishev introduced in \cite{beklemishev2004provability}. We will formally describe such a notation system momentarily, but, roughly, the elements of such notation systems are theories axiomatized by reflection principles and the ordering on them is given by consistency strength. Beklemishev endorsed the use of such notation systems as an approach to the well-known \emph{canonicity problem} of ordinal notation systems. Since then, such notation systems have been intensively studied; see \cite{fernandez2016worms} for a survey of these notation systems and their properties.



We will consider ordinal notation systems based on the calculus $\mathsf{RC}^0$ due to Beklemishev \cite{beklemishev2012calibrating}. In earlier works, e.g. \cite{beklemishev2004provability} on modal logic based ordinal analysis, ordinal notation systems arose from fragments of the polymodal provability logic $\mathsf{GLP}$. However, this application of polymodal provability logic didn't required the full expressive power of $\mathsf{GLP}$. Thus, starting from a work of Dashkov \cite{dashkov2012positive}, strictly positive modal logics have been isolated that yield the same ordinal notation system as the logic $\mathsf{GLP}$, but are much simpler from a technical point of view.

The set of formulas of $\mathsf{RC}^0$ is given by the following inductive definition:
$$F::= \top \;|\; F\land F \;|\; \Diamond_n F\mbox{, where $n$ ranges over $\mathbb{N}$.}$$
An $\mathsf{RC}^0$ sequent is an expression $A\vdash B$, where $A$ and $B$ are $\mathsf{RC}^0$-formulas. The axioms and rules of inference of $\mathsf{RC}^0$ are:
\begin{enumerate}
\item $A\vdash A$; $A\vdash \top$; if $A\vdash B$ and $B\vdash C$ then $A\vdash C$;
\item $A\land B\vdash A$; $A\land B\vdash B$; if $A\vdash B$ and $A\vdash C$ then $A\vdash B\land C$;
\item if $A\vdash B$ then $\Diamond_n A\vdash \Diamond_n B$, for all $n\in\mathbb{N}$;
\item $\Diamond_n\Diamond_nA\vdash \Diamond_nA$, for every $n\in\mathbb{N}$;
\item $\Diamond_n A\vdash \Diamond_m A$, for all $n>m$;
\item \label{RC_Ax6}$\Diamond_n A \land \Diamond_m B \vdash \Diamond_n (A\land \Diamond_m B)$, for all $n>m$.  
\end{enumerate}

Let us describe the intended interpretation of $\mathsf{RC}^0$-formulas in $\mathcal{L}_1$-sentences. The interpretation $\top^{*}$ of $\top$ is $0=0$. The interpretation $(A\land B)^{*}$  is $A^{*}\land B^{*}$. The interpretation $(\Diamond_nA)^{*}$  is $\Ref{\Sigma_{n}}{\mathsf{ACA}_0}(A^{*})$.
A routine check by induction on the length of $\mathsf{RC}^0$-derivations shows that if $A\vdash B$ then $\mathsf{EA}+A^*\vdash B^*$, for any $\mathsf{RC}^0$-formulas $A$ and $B$. 

For a more extensive coverage of positive provability logic see \cite{beklemishev2014positive}.

We denote by $\mathcal{W}$ the set of all $\mathsf{RC}^0$ formulas. The binary relation $<_n$, and the natural equivalence relation $\sim$ are given by
$$A<_nB \stackrel{\mbox{\footnotesize \textrm{def}}}{\iff} B\vdash \Diamond_n A,\;\;\;\;\; A\sim B \stackrel{\mbox{\footnotesize \textrm{def}}}{\iff} B\vdash A\mbox{ and }A\vdash B.$$
The Beklemishev ordinal notation system for $\varepsilon_0$ is the structure $(\mathcal{W}/{\sim},<_0)$.

The following result is due to Beklemishev (see \cite{beklemishev2005veblen,beklemishev2012calibrating}):
\begin{theorem}\label{Beklemishev_ordinal_notation}  $(\mathcal{W}/{\sim},<_0)$ is a well-ordering with the order type $\varepsilon_0$.
\end{theorem}
The transitivity of $(\mathcal{W}/{\sim},<_0)$ is trivial. The linearity of $(\mathcal{W}/{\sim},<_0)$ is provable by a purely syntactical argument within the system $\mathsf{RC}^0$. But Beklemishev's proof of the well-foundedness of $(\mathcal{W}/{\sim},<_0)$ was based on the construction of an isomorphism with Cantor's ordinal notation system for $\varepsilon_0$, i.e., Cantor normal forms.

Here we will give a proof of the well-foundedness part of Theorem \ref{Beklemishev_ordinal_notation} by providing an alternative interpretation of the $\Diamond_n$'s by reflection principles in \emph{second}-order arithmetic and then applying the results of \textsection{\ref{dssection}} to derive well-foundedness.

\begin{theorem} $(\mathcal{W},<_0)$ is a well-founded relation.
\end{theorem}
\begin{proof}
  We prove that the set $\mathcal{W}$ of $\mathsf{RC}^0$-formulas is well-founded with respect to $<_0$. 
  
  We give an alternative interpretation of $\mathsf{RC}^0$. According to this interpretation, the image $\top^{*}$ of $\top$ is $0=0$, $(A\land B)^{*}$  is $A^{*}\land B^{*}$, and $(\Diamond_nA)^{*}$ is $\Ref{\Pi^1_{n+1}}{\mathsf{ACA}_0}(\mathsf{ACA}_0+A^{*})$. 

  We note that if $A\vdash B$ is a derivable $\mathsf{RC}^0$-sequent then $\mathsf{ACA}_0+A^*\vdash B^*$. This can be checked by a straightforward  induction on $\mathsf{RC}^0$-derivations. Also from the definition it is clear that for any $A$ the theory $\mathsf{ACA}_0+A^{*}$ is $\Pi^1_1$-sound (and in fact true $A^*$ is true).

  Now assume for a contradiction that there is an infinite descending chain  $A_0>_0 A_1>_0\ldots$ of $\mathsf{RC}^0$-formulas. Then $A_0^{*},A_1^{*},\ldots$ is an infinite sequence of sentences such that $\mathsf{ACA}_0+A_i^{*}\vdash \mathsf{RFN}_{\Pi^1_{1}}(\mathsf{ACA}_0+A_{i+1}^{*})$. Henceforth we have a $\prec_{\Pi^1_1}$-descending chain of $\Pi^1_1$-sound extensions of $\mathsf{ACA}_0$, contradicting Theorem \ref{well-foundedness_reflection}.
\end{proof}

The key fact that we have used in this proof is that all the theories $A_i^{*}$ are $\Pi^1_1$-sound. In fact all the theories under consideration are subtheories of $\mathsf{ACA}$ and hence the proof is naturally formalizable in $\mathsf{ACA}_0+\mathsf{RFN}_{\Pi^1_1}(\mathsf{ACA}).$\footnote{The fact that $\mathsf{ACA}\equiv_{\Pi^1_\infty}\mathsf{RFN}_{\Pi^1_\infty}(\mathsf{ACA}_0)$ could be proved by a standard technique going back to Kreisel and L{\'e}vy \cite{kreisel1968reflection}. A study of the exact correspondence between restrictions of the schemes of reflection and induction in the setting of second order arithmetic has been recently performed by Frittaion \cite{frittaion2019uniform}.}

Now we show that the same kind of argument could be carried in $\mathsf{ACA}_0$ itself.
\begin{theorem} For each $A\in \mathcal{W}$, the theory $\mathsf{ACA}_0$ proves that $(\{B\in\mathcal{W}\mid B<_0 A\},<_0)$ is well-founded. 
\end{theorem}
\begin{proof} Note that in $\mathsf{RC}^0$ any formula $A$ follows from  formulas $\Diamond_n\top$ such that, for all $\Diamond_m$ that occur in $A$, $m<n$; this fact could be proved by a straightforward induction on length of $A$. Clearly, for any such $n$, the set $\{B\in\mathcal{W}\mid B<_0 A\}$ is a subset of $\{B\in\mathcal{W}\mid B<_0 \Diamond_n\top\}$. Thus, without loss of generality, we may conside only the case of $A$ being of the form $\Diamond_n\top$.
  
  Now we reason in $\mathsf{ACA}_0$. We  assume for a contradiction that there is an infinite descending chain  $\Diamond_n\top>_0A_0>_0 A_1>_0\ldots$ of $\mathsf{RC}^0$-formulas.

  We construct a countably-coded $\omega$-model $\mathfrak{M}$ of $\mathsf{RCA}_0$ that contains this chain. Note that using arithmetical comprehension we could construct a (set encoding) partial satisfaction relation for $\mathfrak{M}$ that the sentence $\mathsf{RCA}_0$ (conjunction of all axioms from some natural finite axiomatization of $\mathsf{RCA}_0$) and all $\Pi^1_{n+1}(\Pi^0_3)$ formulas. We want to show that if $\mathsf{RCA}_0$ proves some $\Pi^1_{n+1}(\Pi^0_3)$ sentence $\varphi$ then $\varphi$ is true in $\mathfrak{M}$. For this we consider any cut-free proof $p$ of the sequent $\lnot\mathsf{RCA}_0,\varphi$. And next by induction on subproofs of $p$ show that all sequents in $p$ are valid in $\mathfrak{M}$ (according to the partial satisfaction relation that we constructed above).  Hence the principle $\mathsf{RFN}_{\Pi^1_{n+1}(\Pi^0_3)}(\mathsf{RCA}_0)$ holds in $\mathfrak{M}$.

We again define an alternative interpretation of $\mathsf{RC}^0$. The interpretation $\top^*$ is $0=0$, the interpretations $(A\land B)^*$ are $A^*\land B^*$, and the intepretations $(\Diamond_k A_i)^*$ are $\mathsf{RFN}_{\Pi^1_{k+1}(\Pi^0_3)}(\mathsf{RCA}_0 + A_i^{*})$. From the previous paragraph we see that $\mathfrak{M}\models (\Diamond_n\top)^{*}$. And since $\Diamond_n\top >_0 A_0$, we have $\mathfrak{M}\models (\Diamond_0 A_0)^*$, i.e., $\mathfrak{M}\models \mathsf{RFN}_{\Pi^1_1(\Pi^0_3)}(\mathsf{RCA}_0+A_0^*)$. Thus in $\mathfrak{M}$ there is an infinite sequence of theories $\mathsf{RCA}_0+A_0^{*}, \mathsf{RCA}_0+A_1^{*},\ldots$ such that $\mathsf{RCA}_0+A_i^{*}\vdash \mathsf{RFN}_{\Pi^1_{1}(\Pi^0_3)}(\mathsf{RCA}_0+A_{i+1}^{*})$ and $\mathsf{RFN}_{\Pi^1_1(\Pi^0_3)}(\mathsf{RCA}_0+A_0^{*})$.  Since $\mathfrak{M}$ is a model of $\mathsf{RCA}_0$, by Theorem \ref{RCA_0_reflection} we reach a contradiction.
\end{proof}


\bibliographystyle{plain}
\bibliography{bibliography}{}

\end{document}